\documentclass[12pt]{article}
\usepackage{hyperref}
\usepackage{tikz}
\usepackage[english]{babel}
\usepackage[utf8]{inputenc}
\usepackage{index}
\usepackage{graphicx}
\usepackage{amsfonts,amsmath,amssymb,amsthm,enumerate,latexsym}

\usetikzlibrary{arrows}
\usetikzlibrary{petri}
\usetikzlibrary{backgrounds}
\usetikzlibrary{decorations.pathreplacing}
\usetikzlibrary{calc}
\tikzstyle{diedge}=[->,shorten <=1pt,>=angle 90,semithick]
\tikzstyle{forced}=[->,shorten <=1pt,>=angle 90,semithick,dashed]
\tikzstyle{bigbad}=[line width=3pt]
\tikzstyle{every token}=[draw=blue!50,fill=blue!20,thick]
\tikzstyle{vert}=[circle,draw=blue!50,fill=blue!20,thick]
\tikzstyle{arw}=[->,shorten <=1pt,>=angle 90,semithick]

\newcommand{\m}[1]{{\mathbf{\uppercase{#1}}}}

\let\phi\varphi

\newcommand\tupl[1]{\overline{#1}}

\let\subset\subseteq

\def\compEXPTIME{{\textsf{EXPTIME}}}

\def\algA{{\mathbf{A}}}
\def\algB{{\mathbf{B}}}

\def\algF{{\mathbf{F}}}

\DeclareMathOperator\Sg{Sg}

\DeclareMathOperator\HasPath{HasPath}
\DeclareMathOperator\TPG{TPG}

\newcommand\vect[1]{\overline{#1}}

\def\en{{\mathbb N}}

\let\epsilon\varepsilon

\theoremstyle{plain}
\newtheorem{theorem}{Theorem}
\newtheorem{lemma}[theorem]{Lemma}
\newtheorem{corollary}[theorem]{Corollary}
\newtheorem{observation}[theorem]{Observation}

\theoremstyle{definition}
\newtheorem{definition}[theorem]{Definition}

\author{Alexandr Kazda, Matt Valeriote}
\title{Deciding some Maltsev conditions in finite idempotent algebras}

\begin{document}
\maketitle
\section{Introduction}

In this paper we investigate the computational complexity of deciding if a
given finite algebraic structure satisfies a certain type of existential
condition on its set of term operations.  These conditions, known as
\emph{Maltsev conditions}, have played a central role in the classification,
study, and applications of general algebraic structures. Several well studied
properties of equationally defined classes of algebras (also known as varieties
of algebras), such as congruence permutability, distributivity, and modularity
are equivalent to particular Maltsev conditions.

The general set up for the decision problems that we consider in this paper is
as follows: for a fixed Maltsev condition $\Sigma$, an instance of
$\Sigma$-testing is a finite algebraic structure (or just algebra for short)
$\m a$.  The question to decide is if the variety of algebras generated by $\m
a$ satisfies the Maltsev condition $\Sigma$. This is a natural computational
problem in universal algebra (especially in the case of important Maltsev
conditions such as having a majority operation) that finds practical
applications in the UACalc~\cite{uacalc} software system for studying algebras
on a computer.  Moreover, some Maltsev conditions are also known to correspond
to complexity classes of the Constraint Satisfaction Problem (see the recent
survey~\cite{barto-krokhin-willard-how-to} for an overview) and deciding these
Maltsev conditions for algebras of polymorphisms of relational structures is an
important meta-problem (which, however, is beyond the scope of this paper).

It turns out that deciding many common Maltsev conditions for finite algebras
is often $\compEXPTIME$-complete~\cite{freese-valeriote, horowitz-ijac} and so
some of the recent work in this area has focussed on the restriction of these
decision problems to finite idempotent algebras. Idempotent algebras display
rich behavior, while also offering relative comfort when it comes to composing
operations. As two examples of the importance of idempotency, we note that
classical Maltsev conditions that characterize properties of congruence
lattices in varieties are idempotent~\cite{Ho-McK}, and the algebraic
approach to the complexity of the Constraint Satisfaction Problem studies
finite idempotent algebras (although the theory gradually moves towards general
algebras, see~\cite{Wonderland-paper}).

Our goal in this paper is to show that $\Sigma$-testing can be
accomplished in polynomial time when the algebras tested are idempotent and the
(strong) Maltsev condition $\Sigma$ can be described using paths.  While not
all Maltsev conditions are of this form, path Maltsev conditions include many
important conditions such as the Maltsev term, ternary majority, Jónsson terms
or Gumm terms. While the efficient decidability of having the Maltsev or
majority term was known before~\cite{freese-valeriote} as was the efficient
decidability of a chain of $n$ Hagemann-Mitschke
terms~\cite{Valeriote_Willard_2014}, our framework of path Maltsev conditions
unifies these earlier results and shows that it is similarly tractable to
decide if an algebra has a fixed length chain of (classical or directed)
Jónsson or Gumm terms. See Corollary~\ref{corMPpoly} for a summary of the
Maltsev conditions we can work with.

While our framework is quite general, it does not cover all previously known
positive results.  For example, in \cite{horowitz-ijac} it is shown that
testing for the presence of a $k$-edge term, for some fixed $k > 1$, can be
done in polynomial time for finite idempotent algebras.  In
\cite{freese-valeriote} it is shown that testing if a finite idempotent algebra
generates a congruence distributive or congruence modular variety can also be
carried out by a polynomial time algorithm. In contrast, the same paper also
proves that testing for either of these conditions in general algebras is an
$\compEXPTIME$-complete problem.

\section{Preliminaries}
An \emph{algebra} is a structure $\m a = \langle A, f_i (i \in I) \rangle$ consisting of a non-empty set $A$, called the \emph{universe} or \emph{domain} of $\m a$, and a list of finitary operations on $A$, $f_i$, for $i \in I$, for some index set $I$, called the \emph{basic operations} of $\m a$. The \emph{type} of $\m a$ is the $I$-indexed sequence $\langle n_i \mid i \in I\rangle$, where the arity of $f_i$ is $n_i$.  A subset $B$ of $A$ is a \emph{subuniverse} of $\algA$ if it is closed under the basic operations of $\algA$.  For any subset $X$ of $A$ there is a smallest subuniverse of $\algA$, with respect to inclusion, that contains $X$.  This subuniverse is called the \emph{subuniverse of $\algA$ generated by $X$} and is denoted by $\Sg^{\algA}(X)$.

A \emph{variety} of algebras of type $\tau$ is a class of algebras of type $\tau$ that is closed under taking homomorphisms, subalgebras, and direct products.   The \emph{variety generated by $\m a$}, denoted by $V(\m a)$, is the smallest variety of algebras having the same type as $\m a$ that contains $\m a$.  For background material on algebras and varieties, the reader may consult one of \cite{Bu-Sa,bergman-book, alvi}.

As noted in the Introduction, we will primarily be concerned with algebras that are finite (their domains are finite and their lists of basic operations are finite) and idempotent.
An operation $t$ on a set $A$ is \emph{idempotent} if the equation
\[
  t(x,x,\dots,x)\approx x.
\]
holds. An algebra $\algA$ is \emph{idempotent} if each of its basic operations is idempotent.

We will  often be working with tuples. The usual notation for an $n$-tuple of
elements of $A$ will be $\vect a\in A^n$. To concatenate two tuples (or,
more often, an $n$-tuple and a single element of $A$) we will use the notation
$\vect a\vect b$ (resp. $\vect a b$ if $b$ is an element of $A$). We will use
the notation $\hat a$ for a tuple of tuples, i.e.. $\hat c=(\tupl c_1,\dots,\tupl
c_n)$ where each $\tupl c_i$ is a tuple.

An \emph{$n$-ary relation} $R$ over a set $A$ is nothing more than a subset of $A^n$.
We will usually write elements of
relations in columns, as this allows us to apply operations of $\algA$ to
$n$-tuples of elements of $A$: Given $\vect r_1,\dots,\vect r_k\in A^n$ and a
$k$-ary operation $f\colon A^k\to A$, the $n$-tuple $f(\vect r_1,\dots,\vect r_k)$ is
the result of applying $f$ to the rows of the matrix with columns $\vect
r_1,\dots,\vect r_k$.
A relation $R$ is \emph{invariant} under $\algA$ if whenever $k\in\en$, $f$
is a $k$-ary operation of $\algA$ and $\vect r_1,\vect r_2,\dots,\vect r_k\in
R$ we have $f(\vect r_1,\dots,\vect r_k)\in R$. A relation $R\leq \prod_{i=1}^n B_i$
is \emph{subdirect} in the product if the $i$-th projection of $R$ is equal to $B_i$ for each $i$.

A good part of our exposition concerns directed graphs, or digraphs. A \emph{digraph} $G$ on the set
of vertices $V(G)$ is a relational structure with one binary relation $E(G)$. We
allow loops in our digraphs and, because our digraphs' edges will have labels on
them, we allow multiple edges between a given pair of vertices.

Graphically, to denote that
 $(u,v)\in E(G)$ for vertices $u$ and $v$, we draw an arrow (oriented edge) from $u$ to $v$. We
say that there is an \emph{oriented walk} from $u\in V(G)$ to $v\in V(G)$ if there is a
sequence of vertices $u=w_0,w_1,w_2,\dots,w_n=v$ such that for all $i$ there is
an edge from $w_i$ to $w_{i+1}$ or from $w_{i+1}$ to $w_{i}$. A walk is
\emph{directed} if for all $i$ the edge is always from $w_i$ to $w_{i+1}$. An
oriented walk where the vertices $w_0,\dots,w_n$ are pairwise distinct is an
\emph{oriented path}. If $P$ is an oriented path with vertices $p_0,\dots,p_n$,
then a \emph{prefix} of $P$ is any oriented path $p_0,\dots,p_k$  and
a \emph{suffix} of $P$ is any oriented path $p_k,\dots,p_n$ where $0\leq k\leq n$.
A \emph{cycle of length $n$} is a sequence of
vertices $w_0,w_1,\dots,w_n$ such that $w_0=w_n$ and for each $i$ we have an
edge from $w_i$ to $w_{i+1}$.

For the purposes of this paper, a
\emph{strong Maltsev condition} $\Sigma$ is a condition of the form: there are some operations $d_1,d_2, \dots, d_k$ that satisfy some set of equations $\Sigma$ involving the $d_i$.
 For example
\begin{equation}\label{MaltsevExample}
    p(p(x,y),r(x))\approx y
\end{equation}
is a strong Maltsev condition involving a binary operation $p(x,y)$ and a unary operation $r(x)$.  An algebra $\m a$ satisfies a strong Maltsev condition $\Sigma$ if for each operation  $s$ that appears in $\Sigma$ there is a term operation $t_s$ of $\m a$ having the same arity as $s$ such that the collection of operations $t_s$ on $\m a$ satisfies the equations of $\Sigma$.

Any Abelian group $\m g = \langle G, x\cdot y, x^{-1}, e\rangle$ satisfies the above strong Maltsev condition, since the term operations $p(x,y) = x\cdot y$ and $r(x) = x^{-1}$ satisfy (\ref{MaltsevExample}).  In contrast, it can be checked that no non-trivial semilattice can satisfy this condition.

A strong Maltsev condition is \emph{linear} if none of its equations involve the composition of operations.  It is said to be \emph{idempotent} if its equations imply that each of the operations involved in it are idempotent.
An example of a strong linear idempotent Maltsev condition is that of having a Maltsev term, i.e., a ternary term $p$ that satisfies the equations $p(y,x,x) \approx y$ and $p(x,x,y) \approx y$.
A more thorough discussion of Maltsev conditions can be found in \cite{Garcia-Taylor}.

\section{Path Maltsev conditions}
In this section, we will show how to express several
classical Maltsev conditions using paths and how to efficiently decide them in finite idempotent algebras by checking that they hold locally.
The proof that one can go from operations that locally satisfy a given path
Maltsev condition in a finite idempotent algebra to operations that satisfy it globally will be presented in the next section.
\subsection{Pattern digraphs}
Inspired by the work of Libor Barto and Marcin Kozik \cite{barto-kozik-boundedwidth-ACM}, we will represent some special Maltsev conditions as paths.
To do this in a systematic way, we will need to introduce digraphs whose edges
carry additional information.

\begin{definition}
A \emph{pattern digraph} is a directed graph with two kinds of edges: Solid and
dashed. Additionally, each pattern digraph has two special distinguished vertices: The initial vertex $s$ and the
terminal vertex $t$ (see Figure~\ref{figPartLabelled}). We allow multiple edges between pairs of vertices of pattern digraphs.
\end{definition}

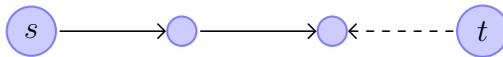
\begin{figure}
  \begin{center}
    \begin{tikzpicture} [node distance=2cm]
      \node[vert](v1){$s$};
      \node[vert, right of=v1](v2){};
      \node[vert, right of=v2](v3){};
      \node[vert, right of=v3](v4){$t$};
      \draw[arw] (v1)--(v2);
      \draw[arw] (v2)--(v3);
      \draw[arw, dashed] (v4)--(v3);
    \end{tikzpicture}
    \caption{An example of a pattern digraph.}\label{figPartLabelled}
  \end{center}
\end{figure}

If $G$ is a pattern digraph, then $H$ is a \emph{subgraph} of $G$ if the set of
vertices, dashed and solid edges of $H$ is the subset of the set of vertices,
dashed and solid edges, respectively of $G$. Additionally, any subgraph of $G$
must have the same initial and terminal vertices as $G$.

A \emph{pattern path} from $s$ to $t$ is a pattern digraph $P$ with initial vertex
$s$ and terminal vertex $t$ such that $P$ viewed as a digraph is an oriented path with
endpoints $s$ and $t$. The length of a pattern path is the number of its edges.

If $G$ and $H$ are pattern digraphs, then a \emph{pattern digraph homomorphism}
(often shortened  to just ``homomorphism'' in the rest of this paper) from $G$ to
$H$ is a mapping $f\colon G\to H$ which is a digraph homomorphism that sends all
solid edges of $G$ to solid edges of $H$.
(Note that  dashed edges of $G$ may be mapped to dashed or solid edges of $H$.)
 We do not require that a pattern digraph homomorphism maps the initial and terminal vertices of its domain to their counterparts in its co-domain. If $G$ is a pattern digraph with
initial vertex $s$ and terminal vertex $t$, and $H$ is a pattern digraph, then
we say that there is a \emph{$G$-shaped walk from $u$ to $v$ in $H$} if there is a
pattern digraph homomorphism $f\colon G\to H$ such that $f(s)=u$ and $f(t)=v$.

We define isomorphisms in the standard way: If $G$ and $H$ are two pattern
digraphs, we say that $G$ and $H$ are isomorphic if and only if there is a
bijection $b\colon G\to H$ such that both $b$ and $b^{-1}$ are pattern digraph
homomorphisms that send initial vertices to initial vertices and terminal vertices
to terminal vertices.

If $G$ and $H$ are pattern digraphs with initial vertices $s_1$ and $s_2$ and
terminal vertices $t_1$ and $t_2$ respectively, then their product $G\times H$ is the pattern
digraph with the vertex set $V(G)\times V(H)$ (of which $(s_1,s_2)$ is the
initial and $(t_1,t_2)$ the terminal vertex). $G\times H$ has an edge from
$(v_1,v_2)$ to $(u_1,u_2)$ if and only if $(v_1,u_1)\in E(G)$ and $(v_2,u_2)\in
E(H)$. The edge $((v_1,v_2),(u_1,u_2))$ is solid if and only if both edges
$(v_1,u_1)$ and $(v_2,u_2)$ are solid; otherwise it is dashed (see
Figure~\ref{figCMtermsProduct}).

\begin{figure}
  \begin{center}
    \begin{tikzpicture} [node distance=2cm]
    \node (times) at (0,0) {$\times$};
    \node [vert, label=$s_1$] (x)at (0,-2) {} ;
    \node [vert, label=below:$t_1$] (y)at (0,-4) {} ;
    \draw[arw,dashed] (x)--(y);
    \draw[arw] (x) edge[loop left] (x);
    \draw[arw] (y) edge[loop left] (y);

    \node[vert, label=$s_2$](v1) at (2,0) {};
      \node[vert, right of=v1](v2){};
      \node[vert, right of=v2](v3){};
      \node[vert, right of=v3,label=$t_2$](v4){};
      \draw[arw] (v1)--(v2);
      \draw[arw] (v2)--(v3);
      \draw[arw, dashed] (v4)--(v3);
    \foreach \i in {2,4,6,8}
       \foreach \j in {-2,-4}
       \node[vert] (p\i\j) at (\i,\j){};
    \foreach \j/\k/\styl in {-2/-2/,-4/-4/,-2/-4/dashed}{
      \draw[arw,\styl] (p2\j)--(p4\k);
      \draw[arw,\styl] (p4\j)--(p6\k);
      \draw[arw,dashed] (p8\j)--(p6\k);
    }
    \node[label=above:{$(s_1,s_2)$}] at (2,-2) {};
    \node[label=below:{$(t_1,t_2)$}] at (8,-4) {};

    \end{tikzpicture}
    \caption{A product of two pattern digraphs.}\label{figCMtermsProduct}
  \end{center}
\end{figure}
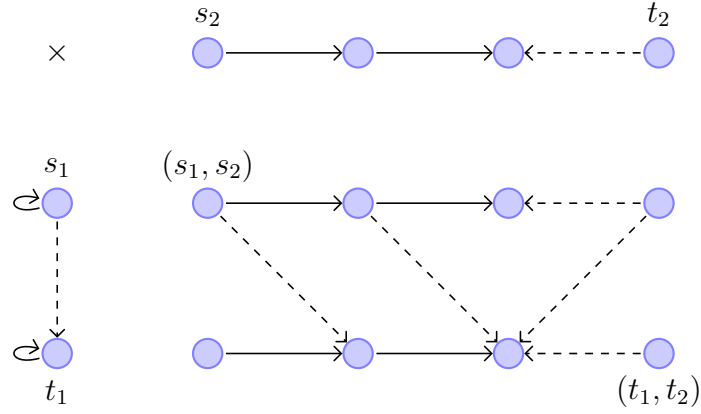

Let $\algA$ be an algebra and let $\algF$ be the 2-generated free algebra in
$V(\algA)$ with generators $x,y$.
The elements of $\algF$ can be represented as binary term operations $b(x,y)$ of the algebra $\algA$ in the variables $x$ and $y$, and from this perspective can be regarded as members of $A^{|A|^2}$, i.e., as $|A|^2$-tuples over $A$.  Under this representation the generators $x$ and $y$ correspond to the two binary projection functions $\pi_0(x,y) = x$ and $\pi_1(x,y) = y$, respectively, and the universe of $\algF$ is the subuniverse of $\algA^{|A|^2}$ generated by $\pi_0$ and $\pi_1$.

 Consider  the following subalgebra  of $\algF^3$:
\[
  K(\algA)=\Sg^{\algF^3}\left(\left\{\begin{pmatrix} x\\ x\\ x\end{pmatrix},
  \begin{pmatrix} y\\x\\y\end{pmatrix},
\begin{pmatrix}x\\y\\y\end{pmatrix} \right\}
  \right)
\]
Most of the time, we shall view $K(\algA)$ as a pattern digraph on the set
$\algF$ by treating each $(a,b,c)\in K(\algA)$ as an edge from $b$ to $c$. If
$a=x$, we declare the edge $(b,c)$ to be solid, otherwise it will be dashed.
The initial and terminal vertices of the pattern digraph $K(\algA)$ are
$x$ and
$y$, respectively.

If we look at the generators of $K(\algA)$ only and view them as a pattern digraph
with initial vertex $x$ and terminal vertex $y$, we will get a digraph that has
two vertices, two solid loops and one dashed edge (see
Figure~\ref{figGenerators}). We will call this digraph $J$.

\begin{figure}
  \begin{center}
    \begin{tikzpicture} [node distance=2cm]
      \node [vert, label=$x$]    (x) {};
      \node [vert, right of= x,label=$y$] (y) {};
      \draw[arw,dashed] (x)--(y);
      \draw[arw] (x) edge[loop below] (x);
      \draw[arw] (y) edge[loop below] (y);
    \end{tikzpicture}
  \end{center}
  \caption{The pattern digraph $J$.}\label{figGenerators}
\end{figure}
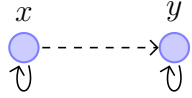

While most of the time we only have to distinguish solid and dashed edges,
sometimes this approach is too coarse. For example, when we
view $K(\algA)$ as a pattern digraph, we are losing information about the first
coordinates of members of $K(\algA)$. This is why we will sometimes be talking about labelled digraphs
instead of pattern digraphs. A \emph{labelled digraph} is a directed graph $G$
together with a function $f\colon E(G)\to L$ that assigns a label from the set
$L$ to each edge of $G$. We will again allow multiple edges between two vertices.

\subsection{From paths to Maltsev conditions}
Let $P$ be a pattern path from $0$ to $n$ with vertex set $\{0,1,2,\dots,n\}$, and let
$\algA$ be an algebra. When does there exist a $P$-shaped walk from
$x$ to $y$ in the pattern digraph $K(\algA)$? We will show that such a path
exists if and only if $\algA$ satisfies a specific strong linear Maltsev condition
$M(P)$.

Given $P$, we will construct $M(P)$ as follows: For each
$i\in\{0,1,\dots,n\}$, take a binary operation symbol $s_i(x,y)$, and for each
$i\in[n] = \{1, 2, \dots, n\}$, take a ternary operation symbol $t_i(x,y,z)$. Start with the equations $s_0(x,y)\approx x$ and $s_n(x,y)\approx y$.
The equations of $M(P)$ that connect the operation symbols $s_i$ and $t_i$ depend on the nature of the edge between $i-1$ and $i$ in $P$: If the edge has the form $(i-1,i)$ (what we call a forward edge), add
to $M(P)$ the pair of equations
\begin{align*}
  t_{i}(x,x,y)&\approx s_{i-1}(x,y)\\
  t_i(x,y,y)&\approx s_i(x,y),
\end{align*}
while if the $i$-th edge of $P$ is a backward edge $(i,i-1)$, we add to $M(P)$
the pair
\begin{align*}
  t_i(x,x,y)&\approx s_i(x,y)\\
  t_{i}(x,y,y)&\approx s_{i-1}(x,y).
\end{align*}
Moreover, whenever the $i$-th edge of $P$ is solid, we add to $M(P)$ the equation
$t_i(x,y,x)\approx x$.

Looking at $M(P)$, we see that it is a strong, linear Maltsev condition.
Observe also that all the terms $t_i$ and $s_i$ have to be idempotent, since if we
set $x=y$, the chain gives us that all of the $t_i(x,x,x)$ and $s_i(x,x)$ are equal to
each other and to $x$.

\begin{observation}
  Let $\algA$ be an algebra and $P$ be a pattern path. Then the idempotent, strong, linear Maltsev condition $M(P)$ is satisfied by $\algA$ if and only if there is a $P$-shaped walk from $x$ to $y$ in $K(\algA)$.
\end{observation}

\begin{proof}
  Suppose that $\algA$ satisfies $M(P)$, as witnessed by the terms $s_i(x,y)$, for $0 \le i \le n$ and $t_i(x,y,z)$ for $1 \le i \le n$.  We claim that then $s_0,s_1,\dots,s_n$ is a $P$-shaped walk from $x$ to $y$ in $K(\algA)$. Obviously, $s_0(x,y)=x$ and $s_n(x,y)=y$. To see that the edge between $s_{i-1}$ and $s_i$ is of the right type, one needs to consider four cases, of which we will only do one in detail: Assume that $(i-1,i)$ is a solid forward edge of $P$. Then
\[
t_i\left(\begin{pmatrix} x\\ x\\ x\end{pmatrix},
  \begin{pmatrix} y\\x\\y\end{pmatrix},
\begin{pmatrix}x\\y\\y\end{pmatrix}\right)=\begin{pmatrix}t_i(x,y,x)\\t_i(x,x,y)\\t_i(x,y,y)\end{pmatrix}
\]
lies in $K(\algA)$. Using the equations of $M(P)$ involving $t_i$, $s_{i-1}$,
and $s_i$, we immediately see that $(s_{i-1}(x,y), s_i(x,y))$ is a solid edge
of $K(\algA)$ and we are done.

Conversely, suppose that $f_0,f_1,\dots,f_n$ is a $P$-shaped walk from $x$ to $y$ in $K(\algA)$.
We will show how to get the terms $t_i, s_i$ from forward
edges of $P$; the construction for backward edges is similar. If $(i-1, i)$ is an
edge of $P$ then $(f_{i-1}, f_i)$  is an edge of $K(\algA)$ and so there is
some ternary term $g_i(x,y,z)$ of $\algA$ such that $g_i(x,x,y) = f_{i-1}$ and
$g_i(x,y,y) = f_i$. The way to satisfy $M(P)$ is then to let
$t_i(x,y,z)$ to be $g_i(x,y,z)$, $s_{i-1}(x,y)$ to be $f_{i-1}$, and
$s_i(x,y)$ to be $f_i$. Since $f_0=x$ and $f_n=y$,
the equations $s_0\approx x$ and $s_n\approx y$ are satisfied. Finally, if $(i-1,i)$ is a solid
edge, then $(f_{i-1}, f_{i})$ must be solid as well, and thus we can demand that $g_i(x,y,x) = x$, satisfying the corresponding condition in $M(P)$.
\end{proof}

\subsection{Example gallery}
To illustrate the connection between a pattern path $P$ and the associated Maltsev condition $M(P)$, we present some well known
Maltsev conditions as $M(P)$ for some pattern paths $P$. Compare the paths in the pictures with
the set of generators of $K(\algA)$ as shown in Figure~\ref{figGenerators}.
To save space, we will replace $s_0$ by $x$ and $s_n$ by $y$ in the equations.

\subsubsection{The trivial case}
If $P$ contains any dashed forward edges $(i,i+1)$ (see Figure~\ref{figTrivial}) then the condition $M(P)$ will
be trivially satisfied by all algebras; one needs only to set $t_j(x,y,z)=x$
for $j< i$, $t_i(x,y,z)  = y$, and $t_j(x,y,z)=z$ for $j>i$.

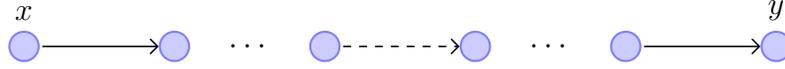
\begin{figure}
  \begin{center}
    \begin{tikzpicture} [node distance=2cm]
      \node [vert, label=$x$]    (a0) {};
      \node [vert, right of =a0] (b1) {};
      \node [vert, right of =b1] (a1) {};
      \node [vert, right of =a1] (b2) {};
      \node [vert, right of= b2]             (bn)  {};
      \node [vert, right of= bn,label=$y$] (bn1) {};
      \coordinate (temp1) at ($(b2)!0.5!(bn)$);
      \node at (temp1) {$\cdots$};
      \coordinate (temp1) at ($(a1)!0.5!(b1)$);
      \node at (temp1) {$\cdots$};
      \foreach \a/\b in {b1/a0,bn1/bn} {
	\draw[arw] (\b) -- (\a);
      }
      \foreach \a/\b in {b2/a1} {
	\draw[arw,dashed] (\b) -- (\a);
      }
    \end{tikzpicture}
  \end{center}
  \caption{A path $P$ yielding trivial $M(P)$.}\label{figTrivial}
\end{figure}

\subsubsection{Maltsev term}
By the classic result due to Maltsev \cite{maltsev}, a variety is congruence permutable if and only if it has a ternary term $t_1(x,y,z)$ that satisfies the equations
\begin{align*}
  x&\approx t_1(x,y,y)\\
  t_1(x,x,y)&\approx y
\end{align*}
 This strong Maltsev condition is equivalent to $M(P)$ for the path $P$ pictured in Figure~\ref{figConPerm} that consists of a single dashed
backward edge $(1,0)$.

\begin{figure}
  \begin{center}
\begin{tikzpicture} [node distance=2cm]
  \node [vert, label=$x$]    (x) {};
  \node [vert, right of= x,label=$y$] (y) {};

  \draw[arw,dashed] (y)--(x);
\end{tikzpicture}
\end{center}
\caption{Maltsev term.}\label{figConPerm}
\end{figure}
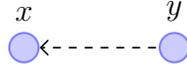

\subsubsection{Majority}
The path $P$ that gives rise to a  majority term consists of a single solid forward edge (see Figure~\ref{figMajority}). The equations that define $M(P)$ are:
\begin{figure}
  \begin{center}
\begin{tikzpicture} [node distance=2cm]
  \node [vert, label=$x$]    (x) {};
  \node [vert, right of= x,label=$y$] (y) {};

  \draw[arw] (x)--(y);
\end{tikzpicture}
\end{center}
\caption{Majority.}\label{figMajority}
\end{figure}
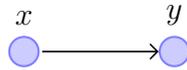
\begin{align*}
  x&\approx t_1(x,x,y)\\
  t_1(x,y,y)&\approx y\\
  t_1(x,y,x)&\approx x.
\end{align*}

\subsubsection{Chain of $n$ J\'onsson terms}\label{subJonsson}
A fence is a sequence of
vertices $v_0,\dots,v_i$ such that $(v_i,v_{i+1})\in E(G)$ for $i$ even and
$(v_{i+1},v_i)\in E(G)$ for $i$ odd.
Classic J\'onsson terms for congruence distributivity \cite{jonsson} arise from  a fence $P$ of $n$ solid edges, for some $n \ge 1$, from $x$ to $y$ that starts with a forward edge. The final edge's direction depends on the parity of $n$;
in the picture (Figure~\ref{figCD2kand1Pattern}),  $n$ is taken to be odd.

The corresponding condition $M(P)$ is:
\begin{align*}
  x&\approx t_1(x,x,y)\\
  t_i(x,y,y)&\approx t_{i+1}(x,y,y)\quad\text{for $1 \le i<n$ odd,}\\
  t_i(x,x,y)&\approx t_{i+1}(x,x,y)\quad\text{for $1 \le i<n$ even,}\\
  t_n(x,y,y)&\approx y\quad\text{for $n$ odd,}\\
  t_n(x,x,y)&\approx y\quad\text{for $n$ even,}\\
  t_i(x,y,x)&\approx x\quad\text{for all $i$.}
\end{align*}
Note that having a single J\'onsson term is the same thing as having a majority term.

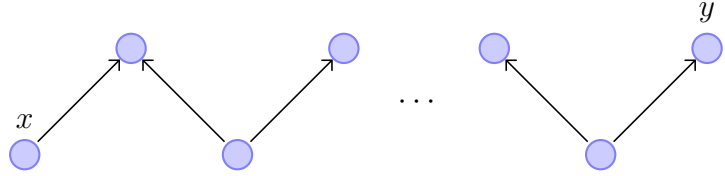
\begin{figure}
  \begin{center}
\begin{tikzpicture} [node distance=2cm]
  \node [vert, label=$x$]    (a0) {};
  \node [vert, above right of =a0] (b1) {};
  \node [vert, below right of =b1] (a1) {};
  \node [vert, above right of =a1] (b2) {};

  \node [vert, right of= b2]             (bn)  {};
  \node [vert, below right of= bn] (an)  {};
  \node [vert, above right of= an,label=$y$] (bn1) {};

  \coordinate (temp1) at ($(a0)!0.5!(b1)$);
  \coordinate (temp2) at ($(b2)!0.5!(bn)$);
  \node at (temp1-|temp2) {$\cdots$};

  \foreach \a/\b in {a0/b1,a1/b1,a1/b2,an/bn,an/bn1} {
    \draw[arw] (\a) -- (\b);
  }
\end{tikzpicture}\caption{A chain of an odd number of J\'onsson terms.} \label{figCD2kand1Pattern}
\end{center}
\end{figure}

\subsubsection{Chain of $n+1$ Gumm terms}
A variety will be congruence modular if and only if it has a finite sequence of Gumm terms \cite{gumm-memoirs}.
These terms are similar to  J\'onsson terms, except that  the last edge ($(n+1)$-st in
our counting) is a backward dashed edge. We note that our formalism
does not capture Day terms \cite{day-terms}, another  chain of terms that captures congruence modularity,
because Day terms have arity four.

Gumm terms are given by the condition $M(P)$, for $P$ the path pictured in Figure~\ref{figGumm}. We denote this condition by CM($n$).
\begin{align*}
  x&\approx t_1(x,x,y)\\
  t_i(x,y,y)&\approx t_{i+1}(x,y,y)\quad\text{for $1 \le i< n$ odd,}\\
  t_i(x,x,y)&\approx t_{i+1}(x,x,y)\quad\text{for $1 \le i<n$ even,}\\
  t_n(x,y,y)&\approx t_{n+1}(x,y,y) \quad\text{for $n$ odd,}\\
  t_n(x,x,y)&\approx t_{n+1}(x,y,y) \quad\text{for $n$ even,}\\
  t_{n+1}(x,x,y)&\approx y\\
  t_i(x,y,x)&\approx x\quad\text{for all $i\leq n$.}
\end{align*}
\begin{figure}
  \begin{center}
\begin{tikzpicture} [node distance=2cm]
  \node [vert, label=$x$]    (a0) {};
  \node [vert, above right of =a0] (b1) {};
  \node [vert, below right of =b1] (a1) {};
  \node [vert, above right of =a1] (b2) {};

  \node [vert, right of= b2]             (bn)  {};
  \node [vert, below right of= bn] (an)  {};
  \node [vert, above right of= an,label=$y$] (bn1) {};

  \coordinate (temp1) at ($(a0)!0.5!(b1)$);
  \coordinate (temp2) at ($(b2)!0.5!(bn)$);
  \node at (temp1-|temp2) {$\cdots$};

  \foreach \a/\b in {a0/b1,a1/b1,a1/b2,an/bn} {
    \draw[arw] (\a) -- (\b);
  }
  \draw[arw,dashed] (bn1)--(an);
\end{tikzpicture} \caption{CM($2k$)}\label{figGumm}
\end{center}
\end{figure}
\subsubsection{Chain of $n$ directed J\'onsson terms}
Directed J\'onsson terms are a variation of those presented in Subsection~\ref{subJonsson} and also can be used to
 characterize congruence distributivity for varieties.
See~\cite{DirectedJonsson} for more details about these terms. The condition $M(P)$ that arises from the path pictured in Figure~\ref{figDirectedJonsson} provides a sequence of directed J\'onsson terms:
\begin{align*}
  x&\approx t_1(x,x,y)\\
  t_i(x,y,y)&\approx t_{i+1}(x,x,y)\quad\text{for $1 \le i < n$} \\
  t_n(x,y,y)&\approx y\\
  t_i(x,y,x)&\approx x\quad\text{for all $i$.}
\end{align*}
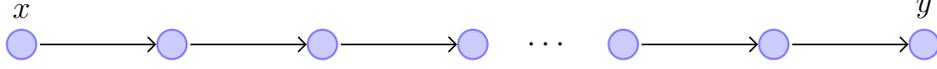
\begin{figure}
  \begin{center}
\begin{tikzpicture} [node distance=2cm]
  \node [vert, label=$x$]    (a0) {};
  \node [vert, right of =a0] (b1) {};
  \node [vert, right of =b1] (a1) {};
  \node [vert, right of =a1] (b2) {};

  \node [vert, right of= b2]             (bn)  {};
  \node [vert, right of= bn] (an)  {};
  \node [vert, right of= an,label=$y$] (bn1) {};

  \coordinate (temp1) at ($(a0)!0.5!(b1)$);
  \coordinate (temp2) at ($(b2)!0.5!(bn)$);
  \node at (temp1-|temp2) {$\cdots$};

  \foreach \a/\b in {a0/b1,b1/a1,a1/b2,bn/an,an/bn1} {
    \draw[arw] (\a) -- (\b);
  }
\end{tikzpicture}
\end{center}
\caption{Directed J\'onsson chain of length $n$}\label{figDirectedJonsson}
\end{figure}

\subsubsection{Chain of directed Gumm terms of length $n+1$}
In a similar manner, one can consider the directed version of  Gumm terms (see~\cite{DirectedJonsson}).
The corresponding path $P$ is pictured in Figure~\ref{figDirectedGumm} and the Maltsev condition $M(P)$ is given by:
\begin{align*}
  x&\approx t_1(x,x,y)\\
  t_i(x,y,y)&\approx t_{i+1}(x,x,y)\quad\text{for $1 \le i < n$} \\
  t_n(x,y,y)&\approx t_{n+1}(x,y,y)\\
  t_{n+1}(x,x,y)&\approx y\\
  t_i(x,y,x)&\approx x\quad\text{for all $i$.}
\end{align*}

\begin{figure}
\begin{tikzpicture} [node distance=2cm]
  \node [vert, label=$x$]    (a0) {};
  \node [vert, right of =a0] (b1) {};
  \node [vert, right of =b1] (a1) {};
  \node [vert, right of =a1] (b2) {};

  \node [vert, right of= b2]             (bn)  {};
  \node [vert, right of= bn] (an)  {};
  \node [vert, right of= an,label=$y$] (bn1) {};

  \coordinate (temp1) at ($(a0)!0.5!(b1)$);
  \coordinate (temp2) at ($(b2)!0.5!(bn)$);
  \node at (temp1-|temp2) {$\cdots$};

  \foreach \a/\b in {a0/b1,b1/a1,a1/b2,bn/an} {
    \draw[arw] (\a) -- (\b);
  }

  \draw[arw,dashed] (bn1)--(an);
\end{tikzpicture}
\caption{Directed Gumm chain of length $n+1$}\label{figDirectedGumm}
\end{figure}
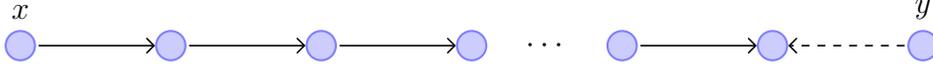
\subsubsection{Chain of $n$ Hagemann-Mitschke terms}
Hagemann-Mitschke terms can be used to characterize varieties that are congruence $(n+1)$-permutable, for a given natural number $n\ge  1$ \cite{Hag-Mit}. The strong Maltsev condition corresponding to this property is given by $M(P)$ for the path $P$ pictured in Figure~\ref{fignperm}:
\begin{align*}
  x&\approx t_1(x,y,y)\\
  t_i(x,x,y)&\approx t_{i+1}(x,y,y)\quad\text{for $1 \le i < n$}  \\
  t_n(x,x,y)&\approx y
\end{align*}

\begin{figure}\begin{tikzpicture} [node distance=2cm]
  \node [vert, label=$x$]    (a0) {};
  \node [vert, right of =a0] (b1) {};
  \node [vert, right of =b1] (a1) {};
  \node [vert, right of =a1] (b2) {};

  \node [vert, right of= b2]             (bn)  {};
  \node [vert, right of= bn] (an)  {};
  \node [vert, right of= an,label=$y$] (bn1) {};

  \coordinate (temp1) at ($(a0)!0.5!(b1)$);
  \coordinate (temp2) at ($(b2)!0.5!(bn)$);
  \node at (temp1-|temp2) {$\cdots$};

  \foreach \a/\b in {a0/b1,b1/a1,a1/b2,bn/an,an/bn1} {
    \draw[arw,dashed] (\b) -- (\a);
  }
\end{tikzpicture}
\caption{Hagemann-Mitschke terms ($n$-permutability)}\label{fignperm}
\end{figure}
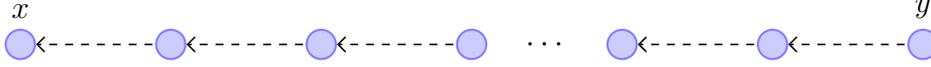

\subsection{The $\HasPath_P$ property}
For a given finite algebra $\algA$, we would like to decide if there is a
$P$-shaped walk from $x$ to $y$ in $K(\algA)$ and hence if $\algA$ satisfies the Maltsev condition $M(P)$. It turns out that for each fixed $P$ there is a polynomial time algorithm
that decides this question as long as $\algA$ is idempotent.
In the rest of this paper, we will assume that $P$ is a pattern path from 0 to $n$ with vertex set $\{0,1,\dots,n\}$ and that $P$ has no dashed forward edges (for else $M(P)$ would be trivial).

As noted earlier, we regard the free algebra $\algF$ as a subuniverse of $\algA^{|A|^2}$. Therefore, the
labelled graph $K(\algA)$ is a subuniverse of the $3|A|^2$-th power of $\algA$
and most likely too large to be searched directly. Our goal in the following is
to approximate $P$-shaped walks in $K(\algA)$ using lower powers of $\algA$.

One issue with our approach is that we need to approximate the images of
different members of $P$ by different subpowers of $\algA$. To facilitate this,
we will use products.

Consider the pattern digraph $K(\algA)\times P$. The vertex set of this pattern digraph
consists of $n+1$ sets of the form $K(\algA)\times\{i\}$, each of which will be
easy to approximate using subpowers of $\algA$. Moreover, it is elementary to show that
there is a $P$-shaped walk from $x$ to $y$ in $K(\algA)$ if and only if there
exists a $P$-shaped walk from $(x,0)$ to $(y,n)$ in $K(\algA)\times P$ (and that each such walk
necessarily sends the $i$-th vertex of $P$ to a pair in $K(\algA)\times \{i\}$).

As an example, consider Figure~\ref{figCMtermsProduct}. The product in this
picture is in fact $J\times P$ for the pattern path $P$ that corresponds to a
chain of $2+1$ directed Gumm terms and $J$ the pattern digraph that records the generators of $K(\algA)$.
To obtain $K(\algA)\times P$,
one has to close the set of edges between $J\times \{i\}$ and $J \times \{i+1\}$
under the operations of $\algA$. This gives us a hint on how to approximate
$K(\algA)\times P$ by smaller digraphs and we formalize this idea in the notion
of testing pattern digraphs:


\begin{definition}\label{defTesting} Let $\algA$ be an algebra and $P$ a
  pattern path of length $n$. Let $(m_0,\dots,m_n;p_1,\dots,p_n)$ be a tuple of natural numbers, let
  \begin{align*}
    (\tupl a_0,\tupl a_1,\dots, \tupl a_n)&=\hat{a}, &(\tupl c_1,\dots,\tupl c_n)&=\hat c,\\
    (\tupl b_0,\tupl b_1,\dots, \tupl b_n)&=\hat{b}, &(\tupl d_1,\dots,\tupl d_n)&=\hat d,
  \end{align*}
  be tuples of tuples of members of $A$ such that $\tupl a_i,\tupl b_i$ are $m_i$-ary and $\tupl c_i,\tupl d_i$ are $p_i$-ary for all applicable values of $i$.

  The \emph{testing $(m_0,\dots,m_n;p_1,\dots,p_n)$-ary pattern digraph for
  $\algA$ and $P$ generated by $(\hat{a}, \hat{b}, \hat{c}, \hat{d})$}, denoted by $\TPG(\hat{a},\hat{b},\hat{c},\hat{d})$, has  vertex set consisting of a disjoint union of sets (which, after
  Barto and Kozik \cite{barto-kozik-boundedwidth-ACM}, we will call potatoes)
  $B_i=\Sg^{\algA^{m_i}}(\{\tupl a_i,\tupl b_i\})$, where $i=0,1,\dots,n$. The vertices $\tupl a_0$ and $\tupl b_n$ are  the initial and terminal vertices, respectively, of $\TPG(\hat{a},\hat{b},\hat{c},\hat{d})$.

  To obtain the edge set of $G = \TPG(\hat{a},\hat{b},\hat{c},\hat d)$, first
  generate, for $i=1,\dots,n$, the labeled edge sets $E_i$ as follows:
  \begin{itemize}
    \item If the $i$-th edge of $P$ is a forward edge, then
      \[
      E_i=\Sg^{\algA^{p_i+m_{i-1}+m_{i}}}\left(\left\{\begin{pmatrix} \tupl c_i\\ \tupl a_{i-1}\\ \tupl a_i\end{pmatrix},
	    \begin{pmatrix}\tupl d_i \\\tupl a_{i-1}\\\tupl b_i\end{pmatrix},
	  \begin{pmatrix}\tupl c_i\\\tupl b_{i-1}\\\tupl b_i\end{pmatrix} \right\}
  \right)	  \leq \algA^{p_i}\times B_{i-1}\times B_i
      \]
    \item If the $i$-th edge of $P$ is a backward edge, then swap $B_{i-1}$ and $B_i$ (as well as the corresponding generators) in the above definition, i.e.
   \[
   E_i=\Sg^{\algA^{p_i+m_{i}+m_{i-1}}}\left(\left\{\begin{pmatrix} \tupl c_i\\ \tupl a_{i}\\ \tupl a_{i-1}\end{pmatrix},
       \begin{pmatrix}\tupl d_i \\\tupl a_{i}\\\tupl b_{i-1}\end{pmatrix},
     \begin{pmatrix}\tupl c_i\\\tupl b_{i}\\\tupl b_{i-1}\end{pmatrix} \right\}
     \right)	  \leq \algA^{p_i}\times B_{i}\times B_{i-1}
   \]
  \end{itemize}
  To obtain the edges of $G$, we translate all members of $\bigcup_{i=1}^n E_i$
  into either solid or dashed edges: Given $(\tupl e,\tupl f,\tupl g)\in E_i$,
  we place into $G$ an edge from $\tupl f$ to $\tupl g$. If $\tupl e=\tupl c_i$
  and the $i$-th edge of $P$ is solid, the new edge of $G$ is solid, otherwise
  the new edge is dashed.
 (As a consequence, the values of $p_i,\tupl c_i,\tupl d_i$ only matter when the
 $i$-th edge of $P$ is solid. When the $i$-th edge is dashed, we will nonetheless
 keep these dummy parameters to make the notation simpler.)
\end{definition}

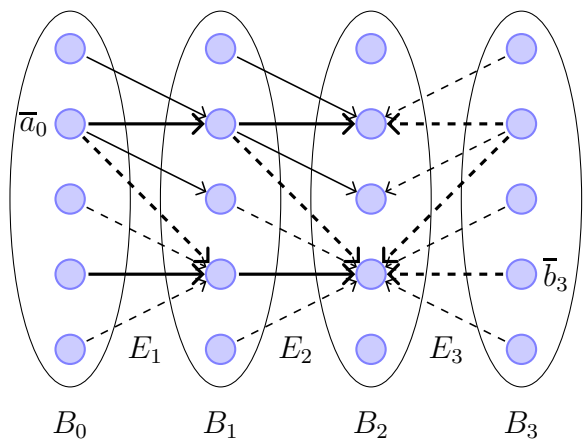
\begin{figure}
  \begin{center}
    \begin{tikzpicture} [node distance=2cm]
      \foreach \i/\k in {2/0,4/1,6/2,8/3}{
         \foreach \j in {-1,-2,-3,-4,-5}{
	   \node[vert] (p\i\j) at (\i,\j){};
         }
         \draw (\i,-3) ellipse (0.8cm and 2.5cm);
	 \node at (\i,-6) {$B_\k$};
       }
       \node[label=left:$\tupl a_0$] at (p2-2) {};
       \node[label=right:$\tupl b_3$]at (p8-4) {};
    \foreach \j/\k/\styl in {-1/-2/,-5/-4/dashed,-2/-3/,-3/-4/dashed}{
      \draw[arw,\styl] (p2\j)--(p4\k);
      \draw[arw,\styl] (p4\j)--(p6\k);
      \draw[arw,dashed] (p8\j)--(p6\k);
    }
    \foreach \i/\k in {3/1,5/2,7/3}
      \node at (\i,-5) {$E_\k$};
    \foreach \j/\k/\styl in {-2/-2/,-4/-4/,-2/-4/dashed}{
      \draw[arw,very thick,\styl] (p2\j)--(p4\k);
      \draw[arw,very thick,\styl] (p4\j)--(p6\k);
      \draw[arw,dashed,very thick,\styl] (p8\j)--(p6\k);
    }
    \end{tikzpicture}
    \caption{An example of a testing pattern digraph $G$ with the generating set drawn thick. The
      path $P$ is the same as in Figure~\ref{figCMtermsProduct}.}\label{figCMGIexample}
  \end{center}
\end{figure}

 See Figure~\ref{figCMGIexample} for an example of a testing pattern digraph
$G$  (and its generating set) in the case when $P$ encodes a chain of $2+1$
Gumm terms.

Note that in $\TPG(\hat a,\hat b,\hat c,\hat d)$ the generators of all $E_i$'s
form a pattern digraph isomorphic to $J\times P$ (where $J$ is the
pattern digraph in Figure~\ref{figGenerators}). This will be important later in
Lemma~\ref{lemMinGen}.

\begin{observation}\label{obsSubdirect} 
  Let $\algA$ be an idempotent algebra, $\tupl m,\tupl p$ be tuples of positive integers
  and $\TPG(\hat{a},\hat{b},\hat{c},\hat{d})$ be an $(\tupl m,\tupl p)$-ary
  testing pattern digraph. Then:
  \begin{enumerate}[(a)]
    \item the set of edges between $B_{i-1}$ and $B_i$ forms a subdirect relation,
    \item if the $i$-th edge of $P$ is solid then the set of solid edges between $B_{i-1}$ and $B_{i}$ is also a subdirect relation,
    \item \label{itmAllToOne} if the $i$-th edge of $P$ is a forward edge then
      $\TPG(\hat a,\hat b,\hat c,\hat d)$ contains an edge from any $\tupl p\in B_{i-1}$ to $\tupl b_i$, and
    \item if the $i$-th edge of $P$ is a backward edge then $\TPG(\hat a,\hat b,\hat c,\hat d)$ contains an edge from any $\tupl p\in B_{i}$ to $\tupl b_{i-1}$.
  \end{enumerate}
\end{observation}
\begin{proof}
  To see the first two claims, observe that the projection of each $E_i$ to
  $A^{p_i}\times B_i$ contains the tuples $(\tupl c_i,\tupl a_i), (\tupl
  c_i,\tupl b_i)$ and $B_i=\Sg(\{\tupl a_i,\tupl
  b_i\})$ (the situation for $B_{i-1}$ is similar).

To prove the third claim, observe that the projection of $E_i$ to $B_{i-1}\times B_i$
contains the tuples $(\tupl a_{i-1},\tupl b_i)$ and $(\tupl b_{i-1},\tupl
b_i)$. Since $\algA$ is idempotent and $B_{i-1}=\Sg(\{\tupl a_{i-1},\tupl
b_{i-1}\})$,
the projection of $E_i$ to $B_{i-1}\times B_i$ contains $B_{i-1}\times \{\tupl b_i\}$. The proof of the last claim is similar.
\end{proof}

A notable example of a testing pattern digraph is $K(\algA)\times P$
itself.
Continuing with our representation of $\algF$ as a subuniverse of $\algA^{|A|^2}$, with the free generators $x$ and $y$ equal to the binary projection maps $\pi_0$ and $\pi_1$ respectively, we have that $K(\algA)\times P$ is isomorphic to the 
$(\overline{|A|^2};\overline{|A|^2})$-ary
testing pattern digraph $\TPG(\hat{a},\hat{b},\hat{c},\hat{d})$, where $\hat{a}=(x,x,\dots,x)$, $\hat{b}=(y,y,\dots,y)$, $\hat{c}=(x,\dots,x)$, and
$\hat{d}=(y,\dots,y)$.
Here, and in the following, $\overline{|A|^2}$ denotes a constant tuple or sequence of an appropriate length with value $|A|^2$ at each entry.
To see this claim, note that all of the potatoes $B_0$, $B_1,\dots,B_n$ will be equal to $F$ and each $E_i$
is either $K(\algA)$ or $K(\algA)$ with its second and third coordinates swapped.

\begin{definition}
We say that an algebra $\algA$ satisfies the condition
$\HasPath_P(\tupl m;\tupl p)$ if whenever $G$ is an $(\tupl m;\tupl p)$-ary testing pattern
digraph for $\algA$ and $P$, there is a $P$-shaped walk from the initial to the
terminal vertex in $G$.
\end{definition}

For example, the $G$ in Figure~\ref{figCMGIexample} fails to have a $P$-shaped walk from $\tupl a_0$ to $\tupl b_3$.

The next observation connects the $\HasPath_P$ property with the Maltsev condition
$M(P)$.

\begin{observation}\label{obsScaffolding}
  Let $P$ be a pattern path of length $n$. Then the following are equivalent:
  \begin{enumerate}
    \item Algebra $\algA$ satisfies $M(P)$\label{itmMP}
    \item There is a $P$-shaped walk from $x$ to $y$ in $K(\algA)$\label{itmPtoH}
    \item There is a $P$-shaped walk from $(x,0)$ to $(y,n)$ in $K(\algA)\times P$\label{itmPtoPH}
    \item Algebra $\algA$ satisfies $\HasPath_P(\tupl m; \tupl p)$
      for all choices of tuples  $\tupl m$ and $\tupl p$.\label{itmHasPath}
    \item Algebra $\algA$ satisfies
      $\HasPath_P(\overline{|A|^2};\overline{|A|^2})$\label{itmHasPathAsquared}.
  \end{enumerate}
\end{observation}
\begin{proof}
  The proof is easy once we unpack the definitions. We already know that the
  first three items are equivalent. Trivially,
  $(\ref{itmHasPath})\Rightarrow(\ref{itmHasPathAsquared})$.

  We show that $(\ref{itmMP})\Rightarrow (\ref{itmHasPath})$ as follows: Take a
  testing pattern digraph $\TPG(\hat{a},\hat{b},\hat{c},\hat{d})$. Since $\algA$
  satisfies $M(P)$, there are binary and ternary terms $s_0,\dots,s_n$  and
  $t_1,\dots,t_n$ that satisfy the equations in $M(P)$. Comparing the equations
  in $M(P)$ with the generators of $E_i$, it is straightforward to verify that
  the sequence $h_0,h_1,\dots,h_n$ defined by $h_i=s_i(\tupl a_i,\tupl b_i)$ is
  a $P$-shaped walk from $\tupl a_0$ to $\tupl b_n$.

  To see that $(\ref{itmHasPathAsquared})\Rightarrow(\ref{itmPtoPH})$, recall
  that $K(\algA)\times P$ can be regarded as an $(\overline{|A|^2};\overline{|A|^2})$-ary testing pattern digraph.  The property $\HasPath_P(\overline{|A|^2};\overline{|A|^2})$ then immediately
  gives us a $P$-shaped walk from $(x,0)$ to $(y,n)$ in $K(\algA)\times P$.
\end{proof}

It turns out that if $\algA$ is idempotent, then the minimum arity of instances that
we need to check to determine if $\algA$ satisfies $M(P)$ is merely $(1,\dots,1;1,\dots,1)$ rather than
$(\overline{|A|^2};\overline{|A|^2})$.
\begin{theorem}\label{thmDimensions}
  For a finite idempotent algebra $\algA$ and pattern path $P$ there is a $P$-shaped walk from $x$ to $y$ in the pattern digraph $K(\algA)$ if and only if $\algA$ satisfies $\HasPath_P(1,\dots,1;1,\dots,1)$.
\end{theorem}

The proof of this theorem will be the goal of the next section. Before we start
proving it, though, let us remark on its significance. The condition
 $\HasPath_P(1,\dots,1;1,\dots,1)$ asserts that we can always satisfy the Maltsev
 condition $M(P)$ locally, and this local satisfiability is something that one
 can verify in time polynomial in $\|\algA\|$, where $\|\algA\|$ is a measure of the size of the algebra $\algA$.  To make this definite, we use the measure from \cite{freese-valeriote}:
\[
 \|\algA\| = \sum_{i = 0}^r k_i|A|^i,
\]
where  $r$ is the largest arity of the basic operations of $\algA$ and $k_i$
is the number of basic operations of $\algA$ of arity $i$, for $0 \le i \le
r$. In our analysis below, we will assume that $\algA$ has at least one at
least unary operation (nontrivial idempotent algebras can't have constant
operations) and hence $\|\algA\|\geq |A|$.

Let us now fix a pattern path $P$ of length $n$ with $k$ solid edges (and $n-k$
dashed edges).  To test if an algebra $\algA$ satisfies
$\HasPath_P(1,\dots,1;1,\dots,1)$, we need to examine all
$(1,\dots,1;1,\dots,1)$-ary testing pattern digraphs and check them for
$P$-shaped walks from the initial to the terminal vertex.

Definition~\ref{defTesting} gives us an algorithmic procedure to generate these
digraphs: For each possible combination of values
$\hat{a},\hat{b},\hat{c},\hat{d}$, generate the sets $B_i, E_i$ and translate
them into edges of $G = \TPG(\hat{a},\hat{b},\hat{c},\hat{d})$. Moreover, for the
$n-k$ dashed edges of $P$, the choice of labels $c_i,d_i$ has no effect on $G$ and
we only need to calculate the second and third coordinates of $E_i$. Omitting
these dummy labels, there are $|A|^{2(n+1)+2k}=|A|^{2n+2k+2}$ many tuples
$\hat{a},\hat{b},\hat{c},\hat{d}$ to consider.

Using Proposition 6.1 of \cite{freese-valeriote} it follows that, given $\hat{a},\hat{b},\hat{c},\hat{d}$, the graph $\TPG(\hat{a},\hat{b},\hat{c},\hat{d})$ can be constructed in time
  \[
  O( (n+1) r\|\algA\| + kr\|\algA\|^3 + (n-k)r\|\algA\|^2) =O(  kr\|\algA\|^3 + (n-k)r\|\algA\|^2) .
   \]
Since $k$ is fixed, if $k=0$ the asymptotic simplifies to $O(r\|\algA\|^2)$ and if $k>0$ the
asymptotic is $O(r\|\algA\|^3)$.

Any given testing pattern digraph $G$ has $O(n|A|)$ vertices and $O(n|A|^2)$
edges (for a given pair of vertices, we need to only remember the ``best''
edge, where solid is better than dashed is better than none), testing for a $P$-shaped
walk from $a_0$ to $b_n$ can be done by standard methods in time $O(n|A|^2)$
which is negligible compared to the time needed to generate $G$. All in all,
deciding if $\HasPath_P(1,\dots,1;1,\dots,1)$ holds can be carried out by an
algorithm whose run-time is $O(r|A|^{2n+2}\|\algA\|^2)$ for $k=0$ and
$O(r|A|^{2n+2k+2}\|\algA\|^3)$ for $k>0$.

\begin{corollary}\label{corMPpoly}
  Let $P$ be a fixed pattern path. The  associated idempotent, strong, linear
  Maltsev condition $M(P)$ can be decided for a finite idempotent algebra
  $\algA$ by an algorithm whose run-time can be bounded by a polynomial in
  $\|\algA\|$.
\end{corollary}
Using Theorem~\ref{thmDimensions} and referring to the list of examples from the example gallery, we immediately obtain polynomial-time algorithms for deciding
some well known strong Maltsev conditions for finite idempotent algebras.
\begin{corollary}  Let $n\ge 1$. Each of the following strong Maltsev conditions can be decided for a finite idempotent algebra $\algA$  by a polynomial-time algorithm with the prescribed run-time.
\begin{enumerate}
  \item Having a sequence of $n$ J\'onsson terms, directed or not, can be decided in time $O(r|A|^{4n+2}\|\algA\|^3)$.
  \item Having a sequence of $n+1$ Gumm terms, directed or not, can be decided in time $O(r|A|^{4n+2}\|\algA\|^3)$.
  \item \cite{Valeriote_Willard_2014} Having a sequence of $n$ Hagemann-Mitschke terms can be decided in time $O(r|A|^{2n+2}\|\algA\|^2)$.
  \item \cite{freese-valeriote} Having a Maltsev term can be decided in time $O(r|A|^4\|\algA\|^2)$.
  \item \cite{freese-valeriote} Having a majority term can be decided in time $O(r|A|^6\|\algA\|^3)$.
\end{enumerate}
\end{corollary}
%
%
%
%
%
%
\section{From local to global}
Let $\algA$ be a finite idempotent algebra that satisfies the condition
\[\HasPath_P(1,1,\dots,1;1,\dots,1).\]
 Our goal is to increase all of the parameters of
$\HasPath_P$, eventually establishing that $\HasPath_P(\overline{|A|^2};\overline{|A|^2})$ holds for $\algA$ and thereby proving, using Observation~\ref{obsScaffolding}, that $\algA$ satisfies $M(P)$.

First observe that by adding dummy coordinates, we can always decrease the
parameters of $\HasPath_P$:
\begin{observation}\label{obsDownward}
  If $\algA$ satisfies
  $\HasPath_P(\tupl m; \tupl p)$,
  $m_i'\leq m_i$, for $0 \le i \le n$, and  $p_i'\leq p_i$, for $1 \le i \le n$, then $\algA$ also
  satisfies  $\HasPath_P(\tupl m'; \tupl p')$.
\end{observation}

\begin{definition}\label{defMin} Let $\algA$ be an algebra, $P$ a pattern path and $\tupl m,\tupl
  p$ be tuples of positive integers. The $(\tupl m;\tupl p)$-ary testing pattern
  digraph $\TPG(\hat{a},\hat{b},\hat{c},\hat{d})$ is \emph{minimal} if no proper
  subgraph of $\TPG(\hat{a},\hat{b},\hat{c},\hat{d})$ is isomorphic to another
  $(\tupl m;\tupl p)$-ary testing pattern digraph for $\algA$ and $P$.
\end{definition}

It follows that every testing pattern digraph contains a
subgraph isomorphic to a minimal testing pattern digraph. Since isomorphism
respects initial and terminal vertices and removing edges or vertices can't
create new $P$-shaped walks, it is enough to look at minimal testing pattern
digraphs.

\begin{observation}\label{obsMin}
  Let $\algA$ be an algebra and $\tupl m,\tupl p$ be tuples of values. If every minimal $(\tupl m;\tupl p)$-ary  testing pattern digraph $\TPG(\hat{a},\hat{b},\hat{c},\hat{d})$ has a $P$-shaped walk from the initial to the terminal vertex, then $\algA$ satisfies $\HasPath_P(\tupl m;\tupl p)$.
\end{observation}

The following two lemmas allow us to comfortably handle minimal testing pattern digraphs.

\begin{lemma}\label{lemMinGen}
  If $\TPG(\hat{a},\hat{b},\hat{c},\hat{d})$ is a minimal testing pattern
  digraph and $g$ is a $J\times P$-shaped walk from $\tupl a_0$
  to $\tupl b_n$ in $\TPG(\hat{a},\hat{b},\hat{c},\hat{d})$ then
  $\TPG(\hat{a},\hat{b},\hat{c},\hat{d})$ is equal to the pattern digraph
  $\TPG(\hat{a}',\hat{b}',\hat{c},\hat{d}')$ where we retain the labels $\hat c$, let
  $\tupl a_i'=g(x,i), \tupl b_i'=g(y,i)$, and choose $\tupl d_i'$ so that
  $(\tupl d_i',g(x,i-1),g(y,i))\in E_i$ (if the $i$-th edge of $P$ is a forward
  edge) or   $(\tupl d_i',g(x,i),g(y,i-1))\in E_i$ (if the $i$-th edge of $P$ is a backward edge).
\end{lemma}
\begin{proof}
  Let $E'_i$ be the $i$-th labelled edge set of $\TPG(\hat{a}',\hat{b}',\hat{c},\hat{d}')$. By the choice of generators, we immediately have $E'_i\subset E_i$ and hence $\TPG(\hat{a}',\hat{b}',\hat{c},\hat{d}')$ is a subgraph of $\TPG(\hat{a},\hat{b},\hat{c},\hat{d})$.  By construction, the testing pattern digraph $\TPG(\hat{a}',\hat{b}',\hat{c},\hat{d}')$ has the same initial and terminal vertices as  $\TPG(\hat{a},\hat{b},\hat{c},\hat{d})$  and so by the minimality of $\TPG(\hat{a},\hat{b},\hat{c},\hat{d})$ we conclude that these two testing patterns are equal.
\end{proof}

The following Lemma will help us to prove Lemmas~\ref{lemInduction1},
\ref{lemInduction2}, and \ref{lemInduction3} by induction on the arity of
testing pattern digraphs. The hypothesis of the Lemma is a bit long, but it
describes something quite natural: Given a testing pattern digraph
$G^0$, we can expand the tuples generating $G^0$ to get a more complicated
testing pattern digraph $G$. It now turns out that we can lift any $P$-shaped
path in $G^0$ to an ``almost path'' in $G$.

\begin{lemma}[Partial lifting]\label{lemLifting}
  Let $\algA$ be an algebra and $P$ be a pattern path of length $n$.
  Let $G^0=\TPG(\hat{a}^0,\hat{b}^0,\hat{c}^0,\hat{d}^0)$ be an $(\tupl m;\tupl p)$-ary
  testing pattern digraph. Let $(\tupl m';\tupl p')$ be  such that $m_i'\geq
  m_i$ for each $i\in\{0,1,\dots,n\}$ and $p_i'\geq p_i$
  for each $i=1,\dots,n$. Let $G=\TPG(\hat{a},\hat{b},\hat{c},\hat{d})$ be any $(\tupl m';\tupl p')$-ary testing pattern digraph
  so that for each applicable $i$ the tuple $\tupl a_i^0$ is a prefix of $\tupl a_i$,
$\tupl b_i^0$ is a prefix of $\tupl b_i$, $\tupl c_i^0$ is a prefix of $\tupl
  c_i$, and $\tupl d_i^0$ is a prefix of $\tupl d_i$.

  Let $p$ be a $P$-shaped path from $\tupl a_0^0$ to $\tupl b_n^0$ in $G^0$.
  Denote by $\tupl e_1,\dots,\tupl e_n$ the labels of the edges of $p$. Then there exist
  tuples $\tupl u_0,\dots, \tupl u_{n}$, $\tupl w_0,\dots,\tupl w_{n}$
  and $\tupl f_1,\dots,\tupl f_n$ so that
  \begin{enumerate}
    \item for each $i=0,1,\dots,n$, the tuple $p(i)$ is a prefix of both
      $\tupl u_i$ and $\tupl w_i$,
    \item for each $i=1,2,\dots,n$, the tuple $\tupl e_i$ is a prefix of $\tupl f_i$,
    \item $\tupl w_0=\tupl a_0$ and $\tupl u_n=\tupl b_n$,
    \item if the $i$-th edge of $P$ is a forward edge, then $(\tupl
      f_i, \tupl u_{i-1}, \tupl w_{i})\in E_i$ (where $E_i$ is the $i$-th
      edge relation of $G$); if the
      $i$-the edge of $P$ is a backward edge then $(\tupl
      f_i, \tupl w_{i},\tupl u_{i-1})\in E_i$,
    \item if $i\in\{0,1,\dots,n\}$ is such that the last $m_i'-m_i$ coordinates of
      $\tupl a_i$ and $\tupl b_i$ agree, then $\tupl u_i=\tupl w_i$,
    \item if $i\in[n]$ is such that the last $p_i'-p_i$ coordinates of
      $\tupl c_i$ and $\tupl d_i$ agree and the $i$-th edge is solid then
      $\tupl f_i=\tupl c_i$.
  \end{enumerate}
  \end{lemma}
\begin{proof}
  Let us examine the $i$-th edge of $p$. Without loss of generality assume
  that the $i$-th edge is a forward edge (the case of backward edges is similar).
  In order for $p$ to be a $P$-shaped path in $G$, we have $(\tupl e_i, p(i-1), p(i)) \in E_i^0$ (where $E_i^0$ is the $i$-th edge relation of $G^0$). Since $E_i^0$ is a subpower of $\algA$ generated
  by the three tuples
  \[
  (\tupl c_{i}^0, \tupl a_{i-1}^0, \tupl a_{i}^0),\  (\tupl d_{i}^0, \tupl a_{i-1}^0, \tupl b_{i}^0),\  (\tupl c_{i}^0, \tupl b_{i-1}^0, \tupl b_{i}^0),
  \]
  there exists a ternary term operation $t_i$ of $\algA$ such
  that
  \begin{align*}
  \tupl e_{i}&=t_i(\tupl c_{i}^0,\tupl d_{i}^0,\tupl c_{i}^0)\\
    p(i-1)&=t_i(\tupl a_{i-1}^0,\tupl a_{i-1}^0,\tupl b_{i-1}^0)\\
    p(i)&=t_i(\tupl a_{i}^0,\tupl b_{i}^0,\tupl b_{i}^0).
  \end{align*}
  We obtain
 $\tupl f_i$, $\tupl u_{i-1}$, and $\tupl w_i$  by extending the input tuples
  (i.e. removing the zero superscripts):
  \begin{align*}
  \tupl f_{i}&=t_i(\tupl c_{i},\tupl d_{i},\tupl c_{i})\\
    \tupl u_{i-1}&=t_i(\tupl a_{i-1},\tupl a_{i-1},\tupl b_{i-1})\\
    \tupl w_i&=t_i(\tupl a_{i},\tupl b_{i},\tupl b_{i}).
  \end{align*}
  This procedure works for $i=1,2,\dots,n$, so it remains to define $\tupl
  w_0=\tupl a_0$ and $\tupl u_n=\tupl b_n$ (see Figure~\ref{fig:almost-path}).
\begin{figure}
  \begin{center}
    \begin{tikzpicture}
      \foreach \i in {0,2,4,6,8}{
	\node[vert] (pp\i) at (\i,4){};
	     }
      \node[left of=pp0]{$P$};
      \foreach \i in {0, 2,4,6, 8}{
	 \foreach \j in {0,1}{
	     \node[vert] (p\i\j) at (\i,\j){};

	   }
	 \draw (\i,0.5) ellipse (0.8cm and 1 cm);
	 }
	 \node[label={[label distance=0.4cm]above:${\tupl a_0=\tupl w_0}$}] at (p01) {};
	 \node[label={[label distance=0.4cm]below:${\tupl b_n=\tupl u_n}$}]at (p80) {};

	 \node[label={[label distance=0.4cm]above:${\tupl w_1}$}] at (p21) {};
	 \node[label={[label distance=0.4cm]above:${\tupl w_2}$}] at (p41) {};
	 \node[label={[label distance=0.4cm]above:${\tupl w_n}$}] at (p81) {};

	 \node[label={[label distance=0.4cm]below:${\tupl u_0}$}] at (p00) {};
	 \node[label={[label distance=0.4cm]below:${\tupl u_1}$}] at (p20) {};
	 \node[label={[label distance=0.4cm]below:${\tupl u_2}$}] at (p40) {};

      \foreach \i/\j in {0/2,4/6}{
         	\draw[arw] (p\i0)--(p\j1);
		\draw[arw] (pp\i)--(pp\j);
	      }
      \foreach \i/\j in {2/4,6/8}{
		\draw[arw] (p\j1)--(p\i0);
		\draw[arw] (pp\j)--(pp\i);
	      }
%
%
    \end{tikzpicture}
  \end{center}
  \caption{The ``almost path'' we obtain by lifting the a $P$-shaped path for
the  $P$  above. Edge labels omitted for simplicity.}\label{fig:almost-path}
\end{figure}


  The conclusions of the lemma
  all follow from the way that the relations $B_i$ and $E_i$ are generated.
 The first two points are consequences of $t_i$ acting coordinatewise. The
  third claim follows from the way we defined $\tupl w_0$ and $\tupl u_n$. Since
  $E_i$ is invariant under $t_i$, we have $(\tupl f_i,\tupl u_{i-1},\tupl
  w_i)\in E_i$, proving the fourth point (the situation for backward edges is
  similar). Finally, the operation $t_i$ is
  idempotent, so if, say, the last $m_i'-m_i$ coordinates of $\tupl a_i$ and
  $\tupl b_i$ are both equal to some tuple $\tupl q$, then
  $\tupl u_i=p(i)\tupl q=\tupl w_i$; the case of edge labels is similar.
\end{proof}

Observe that Lemma~\ref{lemLifting} also holds if we replace prefixes by
suffixes everywhere; we will sometimes use it in this way.

We will only need Lemma~\ref{lemLifting} in the case when $\tupl m',\tupl p'$ differ
from $\tupl m,\tupl p$ in at most one position (either $m_k'>m_k$ or $p_k'>p_k$). Then $m_i'-m_i=0$ and
$p_i'-p_i=0$ for almost all $i$. Therefore, the last two claims of
Lemma~\ref{lemLifting} give us
that $\tupl w_0,\tupl w_1,\dots,\tupl w_k,\tupl u_{k+1},\dots, u_n$ fails to be
a $P$-shaped path only because either $\tupl w_k\neq \tupl u_{k}$,
or $\tupl f_k\neq \tupl c_k$ (if the $k$-th edge is solid).

\begin{lemma}\label{lemPQRS}
Let $\TPG(\hat{a},\hat{b},\hat{c},\hat{d})$ be an
  $(\tupl m; \tupl p)$-ary testing pattern digraph for $\algA$ and $P$ for some tuples $\tupl m$ and $\tupl p$.
  Let $Q$ be a prefix of $P$ and $S$ a suffix of $P$ (possibly overlapping with $Q$). Then:
  \begin{enumerate}[(a)]
    \item\label{itmQ} Every $Q$-shaped walk $q\colon Q\to \TPG(\hat{a},\hat{b},\hat{c},\hat{d})$ from
      $\tupl a_0$ to some $\tupl e_k\in B_k$ can be extended to a
      $J\times P$-shaped walk from $\tupl a_0$ to $\tupl b_n$. (Extending $q$ means that the image of $(x,i)$ is $q(i)$ for all $i=0,1,\dots,k$.)
    \item\label{itmS} Every $S$-shaped walk $s\colon S\to \TPG(\hat{a},\hat{b},\hat{c},\hat{d})$ from
      some $\tupl f_\ell\in B_\ell$ to $\tupl b_n$  can be extended to a
      $J\times P$-shaped walk from $\tupl a_0$ to $\tupl b_n$. (Extending $s$ means that the image of $(y,i)$ is $s(i)$ for all $i=\ell,\dots,n$.)
    \item \label{itmPQRS3} If $\TPG(\hat{a},\hat{b},\hat{c},\hat{d})$ is minimal, then for every $Q$-shaped walk
      $q$ and $S$-shaped walk $s$ as in the previous two parts there exists a
      $J\times P$-shaped walk from $\tupl a_0$ to $\tupl b_n$ in $\TPG(\hat{a},\hat{b},\hat{c},\hat{d})$
      that extends both $q$ and $s$ at the same time (see Figure~\ref{figPQRS}).
  \end{enumerate}
\end{lemma}
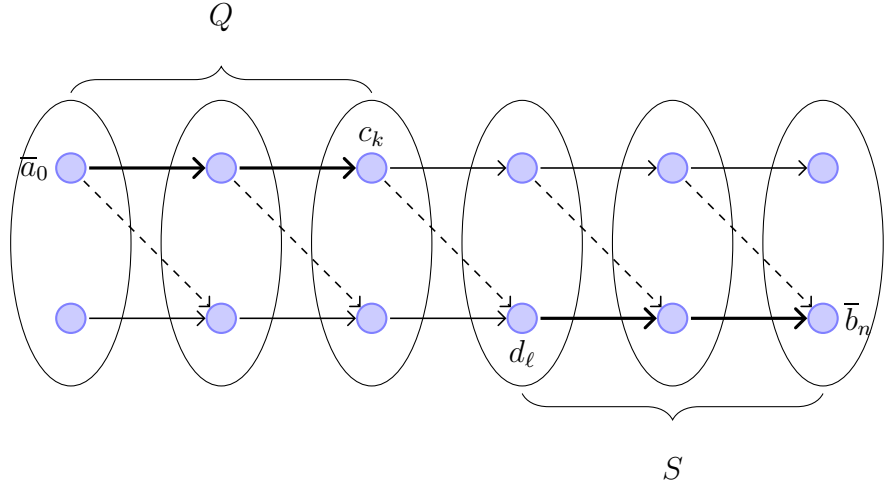
\begin{figure}
  \begin{center}
    \begin{tikzpicture}
      \foreach \i/\k in {2/0,4/1,6/2,8/3,10/4,12/5}{
	 \foreach \j in {0,2}{
	     \node[vert] (p\i\j) at (\i,\j){};
	   }

	 \draw (\i,1) ellipse (0.8cm and 1.9cm);
	 }
	 \node[label=left:$\tupl a_0$] at (p22) {};
	 \node[label=right:$\tupl b_n$]at (p120) {};

	 \node at (4,4) {$Q$};
	 \node at (10,-2) {$S$};
	 \node[label=above:$e_k$] at (p62) {};
	 \node[label=below:$f_\ell$] at (p80) {};
	 \draw [decorate,decoration={brace,amplitude=10pt}]
	 ($(p22)+(0,1)$)--($(p62)+(0,1)$);
         \draw [decorate,decoration={brace,amplitude=10pt}]
	 ($(p120)+(0,-1)$)--($(p80)+(0,-1)$);

      \foreach \j/\k/\styl in {0/0/, 2/2/very thick, 2/0/dashed}{
	\draw[arw,\styl] (p2\j)--(p4\k);}

      \foreach \j/\k/\styl in {0/0/, 2/2/very thick, 2/0/dashed}{
       \draw[arw,\styl] (p6\j)--(p4\k);}

      \foreach \j/\k/\styl in {0/0/, 2/2/, 2/0/dashed}{
	 \draw[arw,\styl] (p6\j)--(p8\k);
       }
      \foreach \j/\k/\styl in {0/0/very thick, 2/2/, 2/0/dashed}{
	 \draw[arw,\styl] (p8\j)--(p10\k);
	 \draw[arw,\styl] (p12\j)--(p10\k);
      }
    \end{tikzpicture}
  \end{center}
  \caption{The situation in part~(\ref{itmPQRS3}) of Lemma~\ref{lemPQRS}. The arrows show the images of
    $J\times P$ under the common extension of $q$ and $s$. Images of the paths $Q$ and $S$
  are drawn in bold.}\label{figPQRS}
\end{figure}
\begin{proof}
  We begin with part~(\ref{itmQ}). By Observation~\ref{obsSubdirect}, we have that the set of
 (solid) edges between $B_i$ and $B_{i+1}$ is subdirect for each $i< n$. This allows the map
 $q$ to be extended to a $P$-shaped walk $q'$ from $\tupl a_0$ to some $q'(n)\in B_n$.
 We claim that $g$ defined by
 $g(x,i)=q'(i)$ and $g(y,i)=\tupl b_i$  for all $0 \le i \le n$ is a $J\times
 P$-shaped walk from $\tupl a_0$ to $\tupl b_n$.

  To prove this, let $0 \le i <n$ and consider the edge in $P$ between $i$ and
  $i+1$. We will examine in detail one of the three possibilities for this edge
  (recall that we disallow forward dashed edges) and leave the other two cases
  for the reader. Suppose that there is a solid edge from $i$ to $i+1$ in $P$.
  We need to prove that $(g(x,i), g(x, i+1))=(q'(i),q'(i+1))$, $(g(y,i), g(y,
  i+1))=(\tupl b_i,\tupl b_{i+1})$, and $(g(x,i), g(y, i+1))=(q'(i),\tupl
  b_{i+1})$ are all edges in $\TPG(\hat{a},\hat{b},\hat{c},\hat{d})$ and that
  the first two are solid.  The first two can be seen to be solid, since $q'$
  is a pattern digraph homomorphism and $(\tupl c_i,\tupl b_i,\tupl b_{i+1})\in
  E_i$, while the existence of the edge $(q'(i),\tupl b_{i+1})$ follows from
  part~(\ref{itmAllToOne}) of Observation~\ref{obsSubdirect}.

  The proof of part~(\ref{itmS}) is similar. Note that our construction for
  this part is such that the image of each $(x,i)$ in the resulting $J\times
  P$-shaped walk is exactly $\tupl a_i$. We will use this in the next
  paragraph.

  We prove part~(\ref{itmPQRS3}) by applying parts~(\ref{itmQ}) and~(\ref{itmS})
  in turn: First get a $J\times P$-shaped walk $g$ from $\tupl a_0$ to $\tupl
  b_n$ that extends $q$. Using Lemma~\ref{lemMinGen}, we get that
  $\TPG(\hat{a},\hat{b},\hat{c},\hat{d})=\TPG(\hat{a}',\hat b',\hat{c},\hat{d}')$,
  where $\tupl a_i'=g(x,i)$ for $i=0,1,\dots,k$. Now apply part~(\ref{itmS}) that we
  have just proved to
  $\TPG(\hat{a}',\hat{b}',\hat{c},\hat{d}')$ and $s$ to obtain a $J\times
  P$-shaped walk from $\tupl a_0$ to $\tupl b_n$ that sends each $(i,x)$ to $\tupl
  a_i'=g(i,x)$ and hence extends $q$ and $s$ at the same time.
\end{proof}

\begin{lemma}\label{lemInduction1}
Suppose that $\algA$ satisfies
  $\HasPath_P(\tupl m; \tupl p)$ for some tuples $\tupl m$ and $\tupl p$ and let $0 \le k \le n$. Then $\algA$
  also satisfies $\HasPath_P(\tupl m^+; \tupl p)$, where $\tupl m^+$ is obtained from $\tupl m$ by increasing $m_k$ by 1.
\end{lemma}
\begin{figure}
  \begin{center}
    \begin{tikzpicture}
      \foreach \i/\k in {2/0,4/1,6/2,8/3,10/4,12/5}{
	 \foreach \j in {0,2}{
	     \node[vert] (p\i\j) at (\i,\j){};
	   }
	 \draw (\i,1) ellipse (0.8cm and 1.9cm);
	 }
	 \node[label=left:$\tupl a_0$] at (p22) {};
	 \node[label=right:$\tupl b_n$]at (p120) {};

	 \node at (4,4) {$Q$};
	 \node at (10,-2) {$S$};
	 \node[label=above:$\tupl w_k$] at (p62) {};
	 \node[label=below:$\tupl u_k$] at (p60) {};
	 \draw [decorate,decoration={brace,amplitude=10pt}]
	 ($(p22)+(0,1)$)--($(p62)+(0,1)$);
         \draw [decorate,decoration={brace,amplitude=10pt}]
	 ($(p120)+(0,-1)$)--($(p60)+(0,-1)$);

      \foreach \j/\k/\styl in {0/0/, 2/2/very thick, 2/0/dashed}{
	 \draw[arw,\styl] (p2\j)--(p4\k);
       \draw[arw,\styl] (p6\j)--(p4\k);
       }
      \foreach \j/\k/\styl in {0/0/very thick, 2/2/, 2/0/dashed}{
	 \draw[arw,\styl] (p6\j)--(p8\k);
	 \draw[arw,\styl] (p10\j)--(p8\k);
	 \draw[arw,\styl] (p10\j)--(p12\k);
      }
    \end{tikzpicture}
  \end{center}
  \caption{The situation in Lemma~\ref{lemInduction1}. The arrows show the image
    of $u$, while the thick arrows are images of $q$ and $s$. Note that $\tupl
    w_k$ and $\tupl u_k$ differ in at most one coordinate.}\label{figInduction1}
\end{figure}

\begin{proof}
  Suppose that
  $G=\TPG(\hat{a},\hat{b},\hat{c},\hat{d})$ is an $(\tupl m^+; \tupl p)$-ary
  testing pattern digraph for $\algA$ and $P$ (all the other testing pattern
  digraphs we will consider in this proof are also for $\algA$ and $P$). We need to construct a
  $P$-shaped walk from $\tupl a_0$ to $\tupl b_n$. By Observation~\ref{obsMin}
  we may assume that $G$ is minimal.

  Let $Q$ be the prefix of $P$ of length $k$ and let $S$ be the suffix of $P$
  of length $n-k$. The vertex sets for $Q$ and $S$ are $\{0, 1, \dots, k\}$ and
  $\{k, k+1, \dots, n\}$ respectively. By forgetting the $(m_k + 1)$-th
  coordinate of $B_k$ we obtain new label sets $\hat{a}^0$ and $\hat{b}^0$.
  Consider the $(\tupl m; \tupl p)$-ary testing pattern digraph $G^0=\TPG(\hat{a}^0,\hat{b}^0,\hat{c},\hat{d})$.

  Applying $\HasPath_P(\tupl m; \tupl p)$ to $G^0$ we obtain a $P$-shaped path
  $p$ from the initial to the terminal vertex of $G^0$. Now apply
  Lemma~\ref{lemLifting} to $G^0$ and $G$. Since $G^0$ and $G$ only differ in the
  $k$-th potato, we see that in $G$ there exists a $Q$-shaped walk
  $q$ from $\tupl a_0$ to some $\tupl w_k\in B_k$ and an $S$-shaped walk $s$
  from a suitable
  $\tupl u_k\in B_k$ to $\tupl b_n$, where $\tupl u_k$ and $\tupl w_k$ differ only in
  the last coordinate (see Figure~\ref{figInduction1}).

  Now, using Lemma~\ref{lemPQRS} on $G$ and paths $p$, $q$, and $s$, we obtain that we can extend both $q$ and $s$
  to a $J\times P$-shaped walk $\alpha$ from $\tupl a_0$ to $\tupl b_n$. We now
  apply Lemma~\ref{lemMinGen} to $G$ and $\alpha$ to show that $G$ is equal to
  the testing pattern digraph  $G'=\TPG(\hat{a}',\hat{b}',\hat{c},\hat{d}')$.
  Examining Lemma~\ref{lemMinGen} in detail, we see that $\tupl a_k'=\alpha(x,k)=\tupl w_k$
  and $\tupl b_k'=\alpha(y,k)=\tupl u_k$.

  If we could show that there is a $P$-shaped path from $\tupl a_0$ to $\tupl
  b_n$ in $G'$, we would be done. To that end, define $\tupl m^-$ to be equal
  to $\tupl m$ everywhere except at the $k$-th place, where $m_k^-=1$, and
  consider the testing pattern digraph
  $G^-=\TPG(\hat{a}^-,\hat{b}^-,\hat{c},\hat{d}')$ of arity $(\tupl m^-,\tupl p)$
  and with $\hat{a}^-,\hat{b}^-$ equal to $\hat{a}',\hat{b}'$ everywhere except
  at the $k$-th position, where we let $\tupl a^-$ and $\tupl b^-$ equal to the
  last entry of $\tupl w_k$ and $\tupl u_k$, respectively. Since $m_k^-\leq
  m_k$, we can apply $\HasPath_P(\tupl m^-,\tupl p)$ to $G^-$.
  We next apply Lemma~\ref{lemLifting} to $G^-$ and $G'$, where we switch from prefixes to suffixes. Noting that the first
  $m_k$ entries of $\tupl a_k'$ and $\tupl b_k'$ agree, we see that the
  sequence of vertices produced by Lemma~\ref{lemLifting} is in fact a
  $P$-shaped path from $\tupl a_0'=\tupl a_0$ to $\tupl b_n'=\tupl b_n$ in
  $G'=G$, concluding the proof.
\end{proof}

\begin{lemma}\label{lemInduction2}
Suppose that $\algA$ satisfies
  $\HasPath_P(\tupl m; \tupl p)$ for some tuples $\tupl m$ and $\tupl p$ and let $1 \le k \le n$. If the $k$-th edge of $P$ is a solid forward edge then  $\algA$
  also satisfies $\HasPath_P(\tupl m; \tupl p^+)$, where $\tupl p^+$ is obtained from $\tupl p$ by increasing $p_k$ by $1$.
\end{lemma}
\begin{proof}
  The proof of this lemma is similar to that of Lemma~\ref{lemInduction1}.
  Again, all testing pattern digraphs will be for $\algA$ and $P$. We
  start with a minimal $(\tupl m; \tupl
  p^+)$-ary testing pattern digraph $G=\TPG(\hat a,\hat b,\hat c,\hat d)$ and will prove that there is a
  $P$-shaped walk from $\tupl a_0$ to $\tupl b_n$ in $G$. We set $Q$ to be the prefix
  of $P$ with the vertex set $\{0, 1, \dots, k-1\}$ and $S$ to be the suffix of
  $P$ with the vertex set $\{k, k+1, \dots, n\}$. Take the $(\tupl m;\tupl
  p)$-ary pattern digraph $G^0=\TPG(\hat a,\hat
  b,\hat c^0,\hat d^0)$ where we get $\hat c^0$ and $\hat d^0$ from $\hat c$
  and $\hat d$ by forgetting the last entry of $\tupl c_k$ and $\tupl d_k$,
  respectively.

  As before, $\HasPath_P(\tupl m; \tupl p)$ and Lemma~\ref{lemLifting} give us that
  there exists an ``almost $P$-shaped'' path from $\tupl a_0$ to $\tupl b_n$ in
  $G$. This time, we have a $Q$-shaped walk $q$ from $\tupl a_0$ to some
  $q(k-1)$ and an $S$-shaped walk $s$ from $s(k)$ to $\tupl b_n$ such that
  $E_k$ contains an edge of the form $(\tupl f_k, q(k-1),s(k))$, where $\tupl
  f_k$ comes from Lemma~\ref{lemLifting}. Were it
  $\tupl f_k=\tupl c_k$, we would have a solid edge from $q(k-1)$ to $s(k)$ and
  we would be done. Alas, it could happen that
  $\tupl f_k$ and $\tupl c_k$ differ in the last coordinate.

  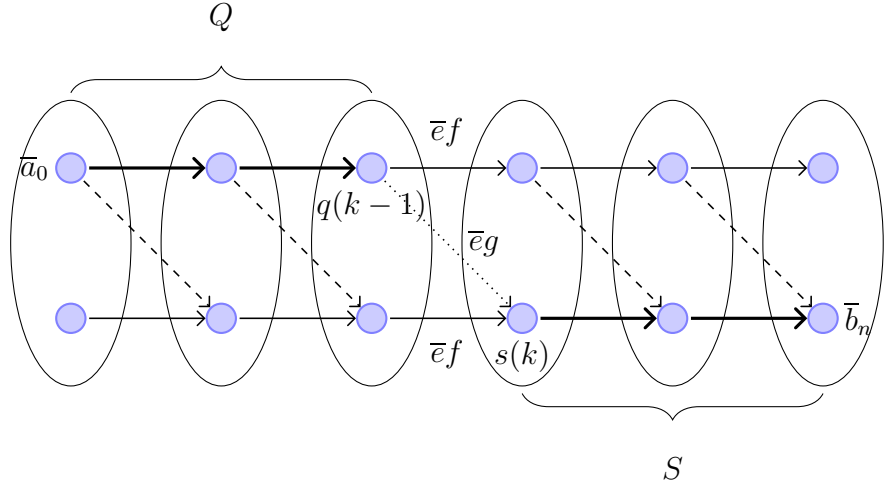
\begin{figure}
  \begin{center}
    \begin{tikzpicture}
	 \foreach \i/\k in {2/0,4/1,6/2,8/3,10/4,12/5}{
	   \foreach \j in {0,2}{
	       \node[vert] (p\i\j) at (1.2*\i,\j){};
	     }
	   \draw[rounded corners] ($(p\i0)-(0.8,1)$) rectangle
	   ($(p\i2)+(0.8,1)$) {};
	 }
	 \node[label=left:$\tupl a_0$] at (p22) {};
	 \node[label=right:$\tupl b_n$]at (p120) {};

	 \node at ($(p42)+(0,2)$) {$Q$};
	 \node at ($(p100)+(0,-2)$) {$S$};
	 \node[label={[label distance=0.0cm]above:$q(k-1)$}] at (p62) {};
	 \node[label=below:$s(k)$] at (p80) {};
	 \node[label=above:$\tupl f_k$] at ($(p80)!0.5!(p62)$){};
	 \node[label=above:$\tupl c_k$] at ($(p82)!0.5!(p62)$){};
	 \node[label=below:$\tupl c_k$] at ($(p80)!0.5!(p60)$){};
	 \draw [decorate,decoration={brace,amplitude=10pt}]
	 ($(p22)+(0,1)$)--($(p62)+(0,1)$);
         \draw [decorate,decoration={brace,amplitude=10pt}]
	 ($(p120)+(0,-1)$)--($(p80)+(0,-1)$);

      \foreach \j/\k/\styl in {0/0/, 2/2/very thick, 2/0/dashed}{
	 \draw[arw,\styl] (p2\j)--(p4\k);
       \draw[arw,\styl] (p6\j)--(p4\k);
       }
      \foreach \j/\k/\styl in {0/0/, 2/2/, 2/0/dotted}{
	 \draw[arw,\styl] (p6\j)--(p8\k);
      }
      \foreach \j/\k/\styl in {0/0/very thick, 2/2/, 2/0/dashed}{
	 \draw[arw,\styl] (p8\j)--(p10\k);
	 \draw[arw,\styl] (p12\j)--(p10\k);
      }
    \end{tikzpicture}
  \end{center}
    \caption{The situation in Lemma~\ref{lemInduction2}. As before, the images
      of $q$ and $s$ are in bold. The labels $\tupl ef$ and $\tupl eg$ refer to
      the labels of the middle edges.}\label{figInduction2}
\end{figure}
We again use part~(\ref{itmPQRS3}) of Lemma~\ref{lemPQRS} to obtain a $J\times
P$-shaped walk from $\tupl a_0$ to $\tupl b_n$ in $G$ that extends $q$ and $s$. Since
we have $(\tupl f_k, q(k-1),s(k))\in E_k$,
Lemma~\ref{lemMinGen} gives us that $G$ is equal to
$G'=\TPG(\hat a',\hat b',\hat c, \hat d')$ for suitable $\hat a'$, $\hat b'$, and
$\hat d'$, with $\tupl d_k'=\tupl f_k$. By forgetting  the first $p_k$
coordinates of the labels of $E_k$ we obtain a $(\tupl m; \tupl p^-)$-ary testing
pattern digraph $G^-$, where $\tupl p^-$ is obtained from $\tupl p$ by replacing $p_k$ by 1.
  $\HasPath_P(\tupl m;p_1,\dots,p_{k-1},1,p_{k+1},\dots,p_{n})$ can be applied to $G^-$ to produce
  a $P$-shaped walk from the initial to the terminal vertex of $G$. Using
  Lemma~\ref{lemLifting} for suffixes instead of prefixes, we can lift this $P$-shaped walk to
  $G'=G$ and are done.
\end{proof}

The case when the $k$-th edge is a backward solid edge will start out
similarly to Lemma~\ref{lemInduction2}, but we will need an additional trick.

\begin{lemma}\label{lemInduction3}
Suppose that $\algA$ satisfies
  $\HasPath_P(\tupl m; \tupl p)$ for some tuples $\tupl m$ and $\tupl p$ and let $1 \le k \le n$. If the $k$-th edge of $P$ is a solid backward edge then $\algA$
  also satisfies $\HasPath_P(\tupl m; \tupl p^+)$, where $\tupl p^+$ is
  obtained from $\tupl p$ by increasing $p_k$ by 1.
\end{lemma}
\begin{proof}

  Take $G$, $G^0$, $Q$, and $S$ as in the proof of Lemma~\ref{lemInduction2}.
  As before we apply $\HasPath_P(\tupl m; \tupl p)$ to $G^0$ and lift the
  result to $G$ via Lemma~\ref{lemLifting}. We get a
  $Q$-shaped walk $q$ from $\tupl a_0$ and an $S$-shaped walk $s$ to $\tupl b_n$ such
  that $(\tupl f_k,s(k),q(k-1)) \in E_k$.
  Part~(\ref{itmPQRS3}) of Lemma~\ref{lemPQRS} again gives us a $J\times P$-shaped walk
  $\alpha$ in $\TPG(\hat a,\hat b,\hat c,\hat d)$ that extends $q$ and $s$.
  Applying Lemma~\ref{lemMinGen}, we again show that $G$ is equal to
  $G'=\TPG(\hat a',\hat b',\hat c,\hat d')$
  where $\tupl a_i'=\alpha(x,i)$ and $\tupl b_i'=\alpha(y,i)$ for each $i$.
  Unfortunately, copying the proof of Lemma~\ref{lemInduction2} fails at this
  point. The problem is that the edge $(\tupl f_k,s(k),q(k-1))$ has the wrong
  endpoints; we would need to have $(\tupl f_k,\alpha(x,k),\alpha(y,k-1))\in
  E_k$ to finish as we did in Lemma~\ref{lemInduction2}. We recover by studying
  an auxiliary relation.
  \begin{figure}
  \begin{center}
    \begin{tikzpicture}

	 \foreach \i/\k in {2/0,4/1}{
	   \foreach \j in {0,2}{
	       \node[vert] (p\i\j) at (\i,\j){};
	     }
	 }
	   \foreach \j in {0,2}{
	       \node[vert, right of=p4\j,node distance=2.5cm] (p6\j){};
	       \node[vert, right of=p6\j,node distance=3.5cm] (p8\j){};
	       \node[vert, right of=p8\j,node distance=2cm] (p10\j){};
	       \node[vert, right of=p10\j,node distance=2cm] (p12\j){};
	     }
          \draw[rounded corners] ($(p60)-(1.2,1)$) rectangle
	   ($(p62)+(1.2,1)$) {};
	   \draw[rounded corners] ($(p80)-(0.8,1)$) rectangle
	   ($(p82)+(0.8,1)$) {};
	   \foreach \i in {2,4,10,12}{
	      \draw[rounded corners] ($(p\i0)-(0.6,1)$) rectangle  ($(p\i2)+(0.6,1)$) {};
	      }

	 \node[label=above:$\tupl a_0$] at (p22) {};
	 \node[label=below:$\tupl b_n$]at (p120) {};

	 \node at ($(p42)+(0,2)$) {$Q$};
	 \node at ($(p100)+(0,-2)$) {$S$};
	 \node at ($(p60)+(0,-1.5)$) {$B_{k-1}$};
	 \node at ($(p80)+(0,-1.5)$) {$B_{k}$};

	 \node[label=above:$q(k-1)$] at (p62) {};
	 \node[label=below:$s(k)$] at (p80) {};
	 \node[label=above:{$\alpha(x,k)$}] at (p82) {};
	 \node[label=below:{$\alpha(y,k-1)$}] at (p60) {};
	 \node[label=above:$\tupl c_k$] at ($(p82)!0.5!(p62)$){};
	 \node[label=below:$\tupl c_k$] at ($(p80)!0.5!(p60)$){};
	 \draw [decorate,decoration={brace,amplitude=10pt}]
	 ($(p22)+(0,1)$)--($(p62)+(0,1)$);
         \draw [decorate,decoration={brace,amplitude=10pt}]
	 ($(p120)+(0,-1)$)--($(p80)+(0,-1)$);

      \foreach \j/\k/\styl in {0/0/, 2/2/very thick, 2/0/dashed}{
	 \draw[arw,\styl] (p2\j)--(p4\k);
       \draw[arw,\styl] (p4\j)--(p6\k);
       }
      \foreach \j/\k/\styl in {0/0/, 2/2/, 0/2/dotted,2/0/dashed}{
	 \draw[arw,\styl] (p8\j)--(p6\k);
      }
      \foreach \j/\k/\styl in {0/0/very thick, 2/2/, 2/0/dashed}{
	 \draw[arw,\styl] (p8\j)--(p10\k);
	 \draw[arw,\styl] (p10\j)--(p12\k);
      }
    \end{tikzpicture}
  \end{center}
    \caption{The situation in the first paragraph of the proof of
    Lemma~\ref{lemInduction3}. As before, the images
      of $q$ and $s$ are in bold. The label $\tupl f_k$ is depicted as a dotted
      edge.}\label{figInduction3-start}
\end{figure}

Let us choose $\tupl e\in A^{p_k}$ and $f,c\in A$ so that
$\tupl f_k=\tupl ef$ and $\tupl c_k=\tupl ec$.
Define the new relation $R\subset A\times B_{k}\times B_{k-1}$ as follows:
  \[
    R=\{(x,\tupl y,\tupl z)\colon \exists \tupl t\in B_{k-1},\tupl v\in B_k,\,(\tupl c_k,\tupl y,\tupl
  t), (\tupl ex,\tupl v,\tupl t),(\tupl c_k,\tupl v,\tupl z)\in E_k\}.
  \]
  Since $\algA$ is idempotent, the relation $R$ is a subuniverse of a power of $\algA$.
  Moreover, $R$ contains the tuples $(c,\alpha(x,k),\alpha(x,k-1))$,
  $(c,\alpha(y,k),\alpha(y,k-1))$, and $(f,\alpha(x,k),\alpha(y,k-1))$ (the last tuple is witnessed
  by $(\tupl c_k,\alpha(x,k),\alpha(x,k-1))$, $(\tupl ef,\alpha(y,k),\alpha(x,k-1))$,
  $(\tupl c_k,\alpha(y,k),\alpha(y,k-1))\in E_k$).

\begin{figure}
  \begin{center}
    \begin{tikzpicture}

     	 \foreach \i/\k in {2/0,4/1}{
	   \foreach \j in {0,2}{
	       \node[vert] (p\i\j) at (\i,\j){};
	     }
	 }
	   \foreach \j in {0,2}{
	       \node[vert, right of=p4\j,node distance=2.5cm] (p6\j){};
	       \node[vert, right of=p6\j,node distance=3.5cm] (p8\j){};
	       \node[vert, right of=p8\j,node distance=2cm] (p10\j){};
	       \node[vert, right of=p10\j,node distance=2cm] (p12\j){};
	     }
          \draw[rounded corners] ($(p60)-(1.2,1)$) rectangle
	   ($(p62)+(1.2,1)$) {};
	   \draw[rounded corners] ($(p80)-(0.8,1)$) rectangle
	   ($(p82)+(0.8,1)$) {};
	   \foreach \i in {2,4,10,12}{
	      \draw[rounded corners] ($(p\i0)-(0.6,1)$) rectangle  ($(p\i2)+(0.6,1)$) {};
	      }

	 \node[label=above:$\tupl a_0$] at (p22) {};
	 \node[label=below:$\tupl b_n$]at (p120) {};

	 \node at ($(p42)+(0,2)$) {$Q$};
	 \node at ($(p100)+(0,-2)$) {$S$};
	 \node at ($(p60)+(0,-1.5)$) {$B_{k-1}$};
	 \node at ($(p80)+(0,-1.5)$) {$B_{k}$};

	 \node[label=above:$q(k-1)$] at (p62) {};
	 \node[label=below:$s(k)$] at (p80) {};
	 \draw [decorate,decoration={brace,amplitude=10pt}]
	 ($(p22)+(0,1)$)--($(p62)+(0,1)$);
         \draw [decorate,decoration={brace,amplitude=10pt}]
	 ($(p120)+(0,-1)$)--($(p80)+(0,-1)$);

      \foreach \j/\k/\styl in {0/0/, 2/2/very thick, 2/0/dashed}{
	 \draw[arw,\styl] (p2\j)--(p4\k);
       \draw[arw,\styl] (p4\j)--(p6\k);
       }
      \foreach \j/\k/\styl in {0/0/, 2/2/, 2/0/dotted}{
	 \draw[arw,\styl] (p8\j)--(p6\k);
      }
      \foreach \j/\k/\styl in {0/0/very thick, 2/2/, 2/0/dashed}{
	 \draw[arw,\styl] (p8\j)--(p10\k);
	 \draw[arw,\styl] (p10\j)--(p12\k);
      }

	 \node[label=above:$f$] at ($(p80)!0.5!(p62)$){};
	 \node[label=above:$c$] at ($(p82)!0.5!(p62)$){};
	 \node[label=below:$c$] at ($(p80)!0.5!(p60)$){};

    \end{tikzpicture}
  \end{center}
    \caption{Replacing $E_k$ by $E_k^\star$ in proof of
    Lemma~\ref{lemInduction3}.}\label{figInduction3-Ekprime}
\end{figure}
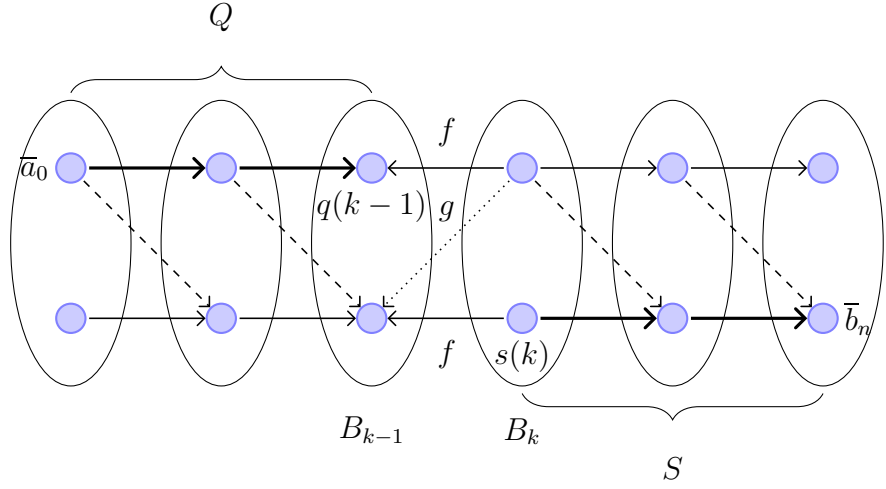
Consider now the $(\tupl m;p_1,\dots,1,\dots,p_n)$-ary testing pattern digraph
$G^\star=\TPG(\hat a',\hat b',\hat c^\star,\hat d^\star)$ where $\hat c^\star$ is $\hat
c$ with $\tupl c_k$ replaced by $c$ and $\hat d^\star$ is $\hat d'$ with $\tupl
d_k$ replaced by $f$. For the most part, this digraph is identical to $\TPG(\hat
a,\hat b,\hat c,\hat d)$, only the edges from $B_{k}$ to $B_{k-1}$ have changed
(see Figure~\ref{figInduction3-Ekprime}). We have $E_k^\star\subset R$ because
the three generators of $E_k^\star$ lie in $R$ (see previous paragraph).

Using $\HasPath(\tupl  m;p_1,\dots,1,\dots,p_n)$ on
$G^\star$, we get a $P$-shaped path $p^\star$
from $\tupl a_0$ to $\tupl b_n$. Since $E_k^\star\subset R$,
there exist $\tupl t\in B_{k-1},\tupl
  v\in B_k$ so that $(p^\star(k),\tupl
  t)$, $(\tupl v,\tupl t)$, and $(\tupl v,p^\star(k-1))$ are all solid edges in the original graph $\TPG(\hat a,\hat b,\hat
  c,\hat d)$. One
  application of part~(\ref{itmPQRS3}) of Lemma~\ref{lemPQRS} with $Q^+$ the
  prefix of $P$ from 0 to $k$ and $S^+$ the suffix of $P$ from $k-1$ to $n$
  (i.e. one edge longer than the original $Q$ and $S$) with $q^+(i)=p^\star(i)$ for
  $i\in \{0,\dots,k-1\}$, $q^+(k)=\tupl v$, and $s^+(i)=p^\star(i)$ for
  $i\in\{k,\dots,n\}$, $s^+(k-1)=\tupl t$ gives us that $q^+$ and $s^+$ can be extended to a
  $J\times P$-shaped walk $\beta$ from $\tupl a_0$ to $\tupl b_n$ (see
  Figure~\ref{figInduction3-end}) in the original $G$.

  It remains to use $\beta$ and Lemma~\ref{lemMinGen} to conclude that $G$ is equal
  to the testing pattern digraph $G^-=\TPG(\hat{a}^-,\hat{b}^-,\hat{c},\hat{d}^-)$
  where $\tupl a_i^-=\beta(x,i)$ and $\tupl b_i^-=\beta(y,i)$ for each
  $i$. Crucially, with $\beta$ in hand we can choose $\tupl d_k^-=\tupl
  c_k$. In other words, it turns out that all edges of the original
  $G$ between $B_{k-1}$ and $B_k$ are solid. But then we can easily connect the
  walks $q$ and $s$ from the first paragraph of this proof into a $P$-shaped
  walk in $G$ and are done.
\end{proof}

 \begin{figure}
  \begin{center}
    \begin{tikzpicture}
     	 \foreach \i/\k in {2/0,4/1}{
	   \foreach \j in {0,2}{
	       \node[vert] (p\i\j) at (\i,\j){};
	     }
	 }
	   \foreach \j in {0,2}{
	       \node[vert, right of=p4\j,node distance=2.5cm] (p6\j){};
	       \node[vert, right of=p6\j,node distance=2.5cm] (p8\j){};
	       \node[vert, right of=p8\j,node distance=2cm] (p10\j){};
	       \node[vert, right of=p10\j,node distance=2cm] (p12\j){};
	     }
          \draw[rounded corners] ($(p60)-(1.2,1)$) rectangle
	   ($(p62)+(1.2,1)$) {};
	   \draw[rounded corners] ($(p80)-(0.8,1)$) rectangle
	   ($(p82)+(0.8,1)$) {};
	   \foreach \i in {2,4,10,12}{
	      \draw[rounded corners] ($(p\i0)-(0.6,1)$) rectangle  ($(p\i2)+(0.6,1)$) {};
	      }

	 \node[label=above:$\tupl a_0$] at (p22) {};
	 \node[label=below:$\tupl b_n$]at (p120) {};

	 \node at ($(p60)+(0,-1.5)$) {$B_{k-1}$};
	 \node at ($(p82)+(0,1.5)$) {$B_{k}$};

      \foreach \j/\k/\styl in {0/0/, 2/2/very thick, 2/0/dashed}{
	 \draw[arw,\styl] (p2\j)--(p4\k);
       \draw[arw,\styl] (p4\j)--(p6\k);
       }
      \foreach \j/\k/\styl in {0/0/, 2/2/, 2/0/dotted}{
	 \draw[arw,\styl] (p8\j)--(p6\k);
      }
      \foreach \j/\k/\styl in {0/0/very thick, 2/2/, 2/0/dashed}{
	 \draw[arw,\styl] (p8\j)--(p10\k);
	 \draw[arw,\styl] (p10\j)--(p12\k);
      }
	 \node at ($(p22)!0.5!(p82)+(0,2)$) {$Q^+$};
	 \node at ($(p120)!0.5!(p60)+(0,-2)$) {$S^+$};
	 \node[label=below:$p^\star(k-1)$] at (p62) {};
	 \node[label=below:$\tupl v$] at (p82) {};
	 \node[label=below:$p^\star(k)$] at (p80) {};
	 \node[label=below:$\tupl t$] at (p60) {};
	 \draw [decorate,decoration={brace,amplitude=10pt}]
	 ($(p22)+(0,1)$)--($(p82)+(0,1)$);
         \draw [decorate,decoration={brace,amplitude=10pt}]
	 ($(p120)+(0,-1)$)--($(p60)+(0,-1)$);

      \foreach \j/\k/\styl in {0/0/, 2/2/, 2/0/dashed}{
	 \draw[arw,\styl] (p2\j)--(p4\k);
       \draw[arw,\styl] (p4\j)--(p6\k);
       }
      \foreach \j/\k/\styl in {0/0/, 2/2/, 2/0/}{
	 \draw[arw,\styl] (p8\j)--(p6\k);
      }
      \foreach \j/\k/\styl in {0/0/, 2/2/, 2/0/dashed}{
	 \draw[arw,\styl] (p8\j)--(p10\k);
	 \draw[arw,\styl] (p10\j)--(p12\k);
      }
    \end{tikzpicture}
  \end{center}
    \caption{Going back to $E_k$ and finishing the proof of
    Lemma~\ref{lemInduction3}. Note that we have three solid edges now.}\label{figInduction3-end}
\end{figure}

\begin{proof}[Proof of Theorem~\ref{thmDimensions}] Start with the tuples
  $\tupl m =(1, 1, \dots, 1)$, $\tupl p = (1, \dots, 1)$, and the hypothesis
  that $\HasPath_P(\tupl m; \tupl p)$ holds. From the definition of testing
  pattern digraphs, we
  know that if the $i$-th edge of $P$ is dashed, then the value of $p_i$ does
  not affect $G$, so we can immediately set any such $p_i$ to $|A|^2$ and
  keep the $\HasPath_P$ property. Next, repeatedly apply
  Lemmas~\ref{lemInduction1}, \ref{lemInduction2}, and \ref{lemInduction3} to
  increase the entries of $\tupl m$ and $\tupl p$ until we get that $\algA$
  satisfies $\HasPath_P(\tupl {|A|^2};\tupl {|A|^2})$. From this it
follows that $\algA$ satisfies $M(P)$.
\end{proof}

\section{Conclusion}
It would be nice if the techniques introduced in proving Theorem~\ref{thmDimensions} could be extended to handle conditions described by some graph (or structure) other than a path and so we ask whether this is the case. One prime example that has been considered is that of a ternary minority operation, i.e., an operation $m(x,y,z)$ that satisfies the equations
\[
m(y,x,x) \approx m (x,y,x) \approx m(x,x,y) \approx y.
\]
While this condition is quite similar to that of having a Maltsev term, it
turns out that there is no chance of producing an efficient algorithm to decide
if a finite idempotent algebra $\algA$ has such a term operation that is based
on some variant of the ``local to global'' term method.  This follows from the
construction, by Dmitriy Zhuk, of a sequence of finite idempotent algebras
$\algA_n$, for $n > 1$, such that for every subset of $\algA_n$ of size $n$,
$\algA_n$ has a term operation that behaves as a minority operation on the
subset, but $\algA_n$ as a whole does not have a minority term
operation~\cite{minority-report}. The
complexity of determining if a finite idempotent algebra has such a term
operation remains open, but recently, the authors, in collaboration with J.
Opr\v{s}al, have shown that this problem is in NP~\cite{minority-report}.

One can also consider the Maltsev conditions of being congruence meet-semidistributive or congruence join-semidistributive.  They are the unions of  sequences of strong linear Maltsev conditions (\cite{kearnes-kiss-book}). We ask whether testing for these strong linear Maltsev conditions can be carried out by polynomial time algorithms for finite idempotent algebras.
It was noted in the introduction that some Maltsev conditions that are not strong, such as being congruence distributive or congruence modular, can also be tested by polynomial time algorithms, for finite idempotent algebras.  We ask whether our techniques can be extended to handle some interesting class of Maltsev conditions that are not strong.

All of the Maltsev conditions studied in this paper are strong, linear, and idempotent (these are called special Maltsev conditions in \cite{Ho-McK}).  As far as we know, for any special Maltsev condition $M$, the problem of deciding if a given finite idempotent algebra satisfies $M$ is in P, and we conjecture that this will always be the case.  An earlier stronger version of this conjecture held that there would be a polynomial time algorithm based on a ``local to global'' term result along the lines of our Theorem~\ref{thmDimensions}, but the case of the minority term Maltsev condition falsified it.


Finally, one can also consider the related problems for finite relational structures.  Namely, given a (strong/linear/idempotent) Maltsev condition $M$ one can ask whether a given finite relational structure $\algB$ has polymorphisms that witness the satisfaction of $M$.  It is known that for strong linear idempotent Maltsev conditions $M$ that imply congruence meet-semidistributivity, this problem can be solved by a polynomial time algorithm.  Recently Hubie Chen and Benoit Larose have produced some significant results related to this class of problems \cite{Chen-Larose}.
\section{Acknowledgments}
A. Kazda was supported by European Research Council under the European Unions
Seventh Framework Programme (FP7/2007-2013)/ERC grant agreement no. 616160 and
the Charles University PRIMUS/SCI/12 and UNCE/SCI/022 grants.

M. Valeriote was supported by a grant from the Natural Sciences and Engineering
Research Council of Canada.
\bibliography{ua}
\end{document}